\documentclass[11pt]{article}

\usepackage[a4paper,margin=1in]{geometry}
\usepackage[T1]{fontenc}
\usepackage[utf8]{inputenc}

\usepackage{amsmath,amsthm,amssymb,amsfonts}
\usepackage{hyperref}
\usepackage{enumitem}
\usepackage[protrusion=true,expansion=false]{microtype}

\hypersetup{
  colorlinks=true,
  linkcolor=blue,
  urlcolor=blue,
  citecolor=blue
}

\numberwithin{equation}{section}

\newtheorem{theorem}{Theorem}[section]
\newtheorem{lemma}[theorem]{Lemma}
\newtheorem{proposition}[theorem]{Proposition}
\newtheorem{corollary}[theorem]{Corollary}
\theoremstyle{definition}
\newtheorem{definition}[theorem]{Definition}
\newtheorem{remark}[theorem]{Remark}
\newtheorem{example}[theorem]{Example}

\title{\textbf{Projection, Measure, and Idempotent Relations:\\
Collapse, Rigidity, and a Fixed-Point Coupling Law}}
\author{
Yunbeom Yi\\[4pt]
\small Independent Researcher, Republic of Korea\\
\small Email: \texttt{28ho2132@gmail.com}
}
\date{}

\begin{document}
\maketitle

\begin{abstract}
We introduce a minimal ZFC--internal axiom system for \emph{pre-structural data}
\[
(X,\mathcal A,\mu,\mu^{\otimes2},R,I,\Pi_R,G,E_0,\eta),
\]
coupling a finitely additive measure $\mu$, an idempotent retraction $\Pi_R:X\to R$, and an idempotent symmetric relation $G\subseteq X\times X$ through a single coupling law (Axiom~III). The central result is a \emph{measure--algebraic collapse theorem}: every admissible structural model is concentrated on the representative sector $R$, namely
\[
\mu(X\setminus R)=0,
\]
with no full--partition hypothesis required. As immediate consequences, the coupling parameter satisfies $\eta<1$ automatically, and the two--point load is rigidly determined,
\[
\mu^{\otimes2}\bigl((B\times X)\cap G\bigr)=\frac{\mu(B)}{1-\eta}\qquad(B\in\mathcal A),
\]
so that it is not an independent datum once $(\mu,\eta)$ are fixed. A further rigidity consequence is component quantization: every measurable $G$--equivalence class $C$ has mass $\mu(C)\in\{0,(1-\eta)^{-1}\}$; as an arithmetic corollary, when finitely many positive--mass classes exhaust the measure their count equals $(1-\eta)E_0$, which must therefore be a positive integer, tying the scale $E_0$ and the rate $\eta$ together. We also establish consistency in ZFC by explicit finite, countable, and continuous (Lebesgue) models, including families with $\eta\neq 0$. We prove mutual independence of the three axioms and of the three subclauses of Axiom~III. Here support collapse is driven by the invariance subclause~III(b) alone; the endpoint exclusion $\eta<1$ follows from the $B=X$ coupling law and the finite product--charge bound, load rigidity additionally uses collapse together with the idempotence of $\Pi_R$ (Axiom~I), and the idempotent--relation axiom~II enters essentially only at component quantization, where it makes $G$ an equivalence relation; the independence results show these ingredients are not formally redundant. Finally we give a fixed--point reformulation of the coupling law as the unique bounded finitely additive solution of a Banach--contraction equation $f=T_\eta f$, and a null--extension factorization exhibiting every admissible model as its identity--retraction core extended by $\mu$--null and $\mu^{\otimes2}$--null data.
\end{abstract}

\bigskip


\medskip


\medskip

{\sloppy\noindent\emph{Note (version 2).}
This version substantially revises and retitles the first version (arXiv:2604.18640v1, \emph{Projection, Measure, and Idempotent Relations: Independent Axioms and a Fixed-Point Coupling Law}). The support--concentration result is strengthened to an \emph{unconditional global collapse} $\mu(X\setminus R)=0$, valid for every admissible model with no full--partition and no fiber--measurability hypothesis (Theorem~\ref{thm:global-collapse}); from it we derive the automatic exclusion $\eta<1$ (Proposition~\ref{prop:eta-lt-one}), universal load rigidity $\mu^{\otimes2}((B\times X)\cap G)=\mu(B)/(1-\eta)$ (Corollary~\ref{cor:rigidity}), measurable component quantization (Theorem~\ref{thm:component-quantization}), and the integrality constraint $(1-\eta)E_0\in\mathbb{Z}_{>0}$ (Corollary~\ref{cor:integrality}). A non--atomic continuous Lebesgue model is added (Proposition~\ref{prop:continuous-model}), the fiber--measurability quotient factorization of v1 is subsumed by a hypothesis--free null--extension theorem (Theorem~\ref{thm:null-extension}), and the Banach fixed--point reformulation is demoted to a corollary of collapse. The countable--model construction is corrected: the diagonal relation $\Delta_Y$ used in v1 need not lie in the finitely generated product algebra $\mathcal A\otimes\mathcal A$ for infinite $Y$, and is replaced by the total relation (Proposition~\ref{prop:countable-model}).\par}


\section{Introduction}

We present a small axiom system, internal to Zermelo--Fraenkel set theory with choice (ZFC; see \cite{Jech,Kunen,Enderton,Levy}), for abstract \emph{pre-structural data}, tuples
\[
\mathcal D=(X,\mathcal A,\mu,\mu^{\otimes2},R,I,\Pi_R,G,E_0,\eta)
\]
consisting of:
(i) a set algebra $(X,\mathcal A)$ with a finitely additive measure $\mu:\mathcal A\to[0,\infty)$ and a finitely additive product charge $\mu^{\otimes2}$ on $\mathcal A\otimes\mathcal A$ satisfying the rectangle rule;
(ii) disjoint measurable sets $R,I\in\mathcal A$ and a map $\Pi_R:X\to R$ (not yet required to be idempotent);
(iii) a measurable relation $G\subseteq X\times X$ (not yet required to be reflexive, symmetric, or idempotent);
(iv) two scalars $E_0>0$ and $\eta\in[0,1]$.

Three axioms tie these components together. Axiom~I imposes projective structure (the retraction property $\Pi_R|_R=\mathrm{id}_R$, which already forces $\Pi_R\circ\Pi_R=\Pi_R$, plus measurability of preimages of measurable subsets of $R$). Axiom~II imposes closure structure on $G$ (reflexivity, symmetry, and the idempotence $G\circ G=G$). Axiom~III imposes quantitative coupling between measure, projection, and relation via three subclauses: conservation ($\mu(R)+\mu(I)=E_0>0$), projection--measure invariance ($\mu(\Pi_R^{-1}(B))=\mu(B)$ for measurable $B\subseteq R$), and a coupling law on $(B\times X)\cap G$. A pre-structural datum satisfying Axioms~I--III is called an \emph{admissible structural model}.

\paragraph{Main phenomenon: admissibility forces collapse.}
The point of this paper is not technical difficulty but structural consequence: a global representative--invariance condition already carries all readable mass onto the retract, and the coupling law then propagates this into load rigidity and quantization. Concretely, although the axioms do not require $X=R\sqcup I$, the projection--measure invariance clause alone forces a global support concentration: every admissible model satisfies
\[
\mu(X\setminus R)=0.
\]
Thus $I$, and indeed every point outside the representative sector $R$, is $\mu$--null. The non--representative part of $X$ may still carry set--theoretic and relational structure, but it is invisible to the measure and to the induced two--point load. This is a genuine strengthening of the elementary observation that, \emph{under} a full partition $X=R\sqcup I$, invariance gives $\mu(I)=0$: here no partition is assumed, only that $R\in\mathcal A$ and $\mathcal A$ is an algebra (Theorem~\ref{thm:global-collapse}).

Three rigidity consequences follow. First, the coupling parameter cannot reach the endpoint: $\eta<1$ is automatic for admissible models (Proposition~\ref{prop:eta-lt-one}). Second, the two--point load carried by $G$ is completely determined,
\[
\mu^{\otimes2}\bigl((B\times X)\cap G\bigr)=\frac{\mu(B)}{1-\eta}\qquad(B\in\mathcal A),
\]
so it ceases to be an independent datum (Corollary~\ref{cor:rigidity}). Third, every measurable $G$--equivalence class is quantized: $\mu(C)\in\{0,(1-\eta)^{-1}\}$ (Theorem~\ref{thm:component-quantization}).

\paragraph{Independence is compatible with collapse.}
This rigidity is not a symptom of redundant axioms. We prove that the three axioms, and the three subclauses of Axiom~III, are mutually independent (Theorem~\ref{thm:independence-main}): no axiom or subclause follows from the others. It is worth being precise about the division of labor. The \emph{support collapse} $\mu(X\setminus R)=0$ is driven by the invariance subclause~III(b) alone (indeed by its single instance $B=R$). The remaining consequences separate by their exact dependencies (Table~\ref{tab:dependencies}): the endpoint exclusion $\eta<1$ follows from the $B=X$ instance of the coupling law~III(c) together with the finite product--charge bound alone; universal load rigidity adds collapse and the idempotence of $\Pi_R$ (Axiom~I); and the idempotent--relation axiom~II becomes essential only at component quantization, where it makes $G$ an equivalence relation. The independence results certify that none of these ingredients is formally redundant: the axioms are independent, yet together they force the full rigidity package. Here ``minimal'' is meant in this independence sense: no axiom, and no subclause of Axiom~III, is derivable from the remaining ones (Theorem~\ref{thm:independence-main}, Corollary~\ref{cor:minimality}). We use finite additivity to keep the axiom system minimal; Section~\ref{sec:normalization} records the role of $\sigma$--additivity, which we show is \emph{not} essential to the normalization obstruction.

\paragraph{On terminology.}
We retain the name \emph{load} for the two--point quantity $\mu^{\otimes2}((B\times X)\cap G)$, occasionally calling it the \emph{idempotent--interaction load}. This is a purely structural label for a bilinear, self--composing interaction pattern; it carries no differential--geometric content. No Riemannian, sectional, or Ricci curvature is defined or used. The idempotence $G\circ G=G$ imposes a closure--type rigidity on pair interactions, and ``load'' simply names the measure of the pairs selected by $G$ over a test set.

\subsection{Motivation}

Projection, closure, and measure are standard objects studied in different areas of mathematics (see e.g.\ \cite{Halmos,Bogachev,AliprantisBorder} for measure theory, \cite{Rao,Fremlin} for charges, \cite{Kuratowski} for topology, \cite{BauschkeCombettes} for projection/retraction operators in analysis, \cite{KolokoltsovMaslov} for idempotent analysis, \cite{KrantzEtAl} for axiomatic foundations of measurement, and \cite{MacLane,Awodey,Riehl} for categorical viewpoints).
The axioms place them into a single framework with direct interaction:
\[
\mu^{\otimes2}((B\times X)\cap G)
= \mu(B) + \eta\,\mu^{\otimes2}((\Pi_R^{-1}(B)\times X)\cap G).
\]
This coupling law does not appear in standard measure--theoretic or closure--theoretic axiomatizations. The emphasis here is not on technical difficulty but on structural consequence: once the projection--measure invariance subclause is also present, the coupling law has only one solution, the load being forced to equal $\mu(B)/(1-\eta)$.

The two structural ingredients are familiar from mainstream measure theory and category theory, which is what makes their rigid interaction noteworthy. A measure--preserving idempotent $\Pi_R$ with $\mu\circ\Pi_R^{-1}=\mu$ on $R$ plays a role analogous to projection mechanisms familiar from conditional expectation, sub--$\sigma$--algebra reduction, and retraction theory \cite{BauschkeCombettes}; the support concentration we prove (Theorem~\ref{thm:global-collapse}) is a finitely additive cousin of the elementary fact that a measure invariant under a measurable retraction is carried by the retract. The idempotent symmetric relation $G$, in turn, sits where an equivalence relation or a measure--algebra projection sits: $G\circ G=G$ together with symmetry and reflexivity makes $G$ an equivalence relation (Proposition~\ref{prop:elementary}), and the two--point load $\mu^{\otimes2}((\,\cdot\,\times X)\cap G)$ measures the pairs it selects. The coupling axiom~III(c) is not an arbitrary postulate joining these two: within the affine self--similar update class encoded by $T_\eta$ --- agreement with $\mu$ at baseline ($\eta=0$), and propagation along $\Pi_R^{-1}$ at rate $\eta$ otherwise --- it is the fixed--point condition whose bounded finitely additive solution is unique. We do not claim uniqueness among all conceivable coupling laws. We make this precise as a Banach fixed--point statement in Section~\ref{sec:fixed-point} (Theorem~\ref{thm:coupling-fixed-point}, Remark~\ref{rem:coupling-not-ad-hoc}); the present collapse results then show that, on an admissible model, even this one degree of freedom is rigidly resolved.

\subsection{Relation to classical frameworks}

Each component has a classical analogue:
\begin{itemize}[nosep]
  \item $G\circ G = G$ encodes an idempotent closure/equivalence--type behavior at the level of relations, parallel to a measure--algebra projection.
  \item $\Pi_R\circ\Pi_R = \Pi_R$ is the defining idempotence of a retraction.
  \item $(X,\mathcal A,\mu)$ is a finitely additive measure space.
  \item The invariance $\mu(\Pi_R^{-1}(B))=\mu(B)$ for $B\subseteq R$ is a measure--preserving retraction onto $R$, in the spirit of measure--preserving maps in ergodic theory.
\end{itemize}
The key point is that these components coexist in one structure, are coupled by Axiom~III, and that their coexistence collapses the structure onto $R$.

\subsection{Contributions}\label{sec:contributions}
The contributions can be summarized as follows.
\begin{itemize}[nosep]
  \item A global support collapse theorem: every admissible model satisfies $\mu(X\setminus R)=0$, hence $\mu(I)=0$ and $\mu(X)=\mu(R)=E_0$, with no full--partition hypothesis (Theorem~\ref{thm:global-collapse}).
  \item Automatic exclusion of the endpoint $\eta=1$ (Proposition~\ref{prop:eta-lt-one}) and a finite--additive feasibility bound $E_0\ge(1-\eta)^{-1}$ (Proposition~\ref{prop:feasibility-bound}).
  \item Universal load rigidity: $\mu^{\otimes2}((B\times X)\cap G)=\mu(B)/(1-\eta)$ for all $B\in\mathcal A$ (Corollary~\ref{cor:rigidity}), proved by an elementary self--substitution independent of the fixed--point machinery.
  \item Measurable component quantization: $\mu(C)\in\{0,(1-\eta)^{-1}\}$ for every measurable $G$--equivalence class (Theorem~\ref{thm:component-quantization}), with the arithmetic corollary that a finite positive--mass block count equals $(1-\eta)E_0\in\mathbb{Z}_{>0}$ (Corollary~\ref{cor:integrality}).
  \item A ZFC--internal axiom system and explicit model families establishing consistency, including finite families with $\eta\neq 0$ and a non--atomic continuous model on $[0,1]$ with Lebesgue measure (Section~\ref{sec:models}, Proposition~\ref{prop:continuous-model}).
  \item Independence of the three axioms and of the three subclauses of Axiom~III (Theorem~\ref{thm:independence-main}, Corollary~\ref{cor:minimality}), showing collapse is genuine rather than a redundancy artifact.
  \item A fixed--point reformulation of the coupling law as the Banach--contraction equation $f=T_\eta f$ with closed form $f_*(B)=\mu(B)+\tfrac{\eta}{1-\eta}\mu(\Pi_R^{-1}(B))$ (Theorem~\ref{thm:coupling-fixed-point}), now recovered as a corollary of collapse.
  \item A null--extension factorization reducing every admissible model to its identity--retraction core (Theorem~\ref{thm:null-extension}); under summability hypotheses the classification reduces to the block dichotomy of Theorem~\ref{thm:pi-id-classification}.
\end{itemize}

\section{Language and basic objects}

We recall standard notions used throughout the paper.

\begin{definition}[Algebra of sets]
Let $X$ be nonempty. A family $\mathcal A\subseteq\mathcal P(X)$ is an \emph{algebra} if $\varnothing,X\in\mathcal A$ and
$B_1,B_2\in\mathcal A$ implies $B_1\cup B_2\in\mathcal A$ and $X\setminus B_1\in\mathcal A$.
\end{definition}

\begin{definition}[Finitely additive measure]
A finitely additive measure is a map $\mu:\mathcal A\to[0,\infty)$ such that $\mu(\varnothing)=0$ and
$\mu(B_1\sqcup B_2)=\mu(B_1)+\mu(B_2)$ whenever $B_1,B_2\in\mathcal A$ are disjoint.
\end{definition}

\begin{definition}[Binary relations and composition]
A binary relation $H$ on $X$ is a subset of $X\times X$. Its composition with a relation $K$ is
\[
H\circ K = \{(x,z):\exists y\ (x,y)\in H,\ (y,z)\in K\}.
\]
The diagonal is $\Delta_X=\{(x,x):x\in X\}$. A relation is \emph{idempotent} if $H\circ H=H$; for composition operations on relations and graphs see \cite{Sabidussi}.
\end{definition}

\begin{definition}[Retraction]
If $R\subseteq X$, a map $\Pi_R:X\to R$ is a \emph{retraction} if $\Pi_R|_R=\mathrm{id}_R$.
In this setting the idempotence $\Pi_R\circ\Pi_R=\Pi_R$ follows automatically: for any $x\in X$ one has $\Pi_R(x)\in R$, so $\Pi_R(\Pi_R(x))=\Pi_R(x)$.
\end{definition}

\begin{definition}[Product algebra and product charge]\label{def:product-measure}
Given $(X,\mathcal A,\mu)$, the product algebra $\mathcal A\otimes\mathcal A$ is the algebra generated by rectangles
$B_1\times B_2$ with $B_i\in\mathcal A$. When $\mathcal A$ is a $\sigma$--algebra and a $\sigma$--additive product is intended, we write $\mathcal A\,\overline{\otimes}\,\mathcal A$ for the product $\sigma$--algebra (the $\sigma$--algebra generated by rectangles), so $\mathcal A\otimes\mathcal A\subseteq\mathcal A\,\overline{\otimes}\,\mathcal A$; the unadorned $\mathcal A\otimes\mathcal A$ always denotes the \emph{algebra}.

A finitely additive set function $\mu^{\otimes2}:\mathcal A\otimes\mathcal A\to[0,\infty)$ is called a \emph{product charge}
(for $\mu$) if it satisfies the rectangle rule
\[
\mu^{\otimes2}(B_1\times B_2)=\mu(B_1)\mu(B_2)\qquad(B_1,B_2\in\mathcal A).
\]
In the finitely additive setting, existence/uniqueness of such an extension is not automatic in full generality; whenever
$\mu^{\otimes2}$ is used below, it is treated as part of the model data satisfying the rectangle rule.
(Under $\sigma$--additivity and $\sigma$--finiteness, the standard product charge exists uniquely; see Section~\ref{sec:normalization}.)
\end{definition}

\begin{lemma}[Atomic product charge exists]\label{lem:atomic-product-charge}
Assume $X$ is finite or countable and $\mathcal A=\mathcal P(X)$. Suppose $\mu$ is atomic, i.e.\ there exists
$m:X\to[0,\infty)$ such that $\mu(B)=\sum_{x\in B}m(x)$ for all $B\subseteq X$. In the countable case assume moreover that the total mass is finite,
\[
\sum_{x\in X}m(x)<\infty,
\]
so that $\mu$ is $[0,\infty)$--valued, consistently with Definition~\ref{def:product-measure}.
Define $\mu^{\otimes2}:\mathcal P(X\times X)\to[0,\infty)$ by
\[
\mu^{\otimes2}(S):=\sum_{(x,y)\in S} m(x)m(y)\qquad(S\subseteq X\times X).
\]
Then $\mu^{\otimes2}$ is finitely additive and satisfies the rectangle rule
$\mu^{\otimes2}(A\times B)=\mu(A)\mu(B)$ for all $A,B\subseteq X$.
In particular, in the finite/countable atomic regimes used in Section~\ref{sec:models}, a product charge can be taken canonically.
\end{lemma}

\begin{proof}
Finite additivity follows from disjointness of sums. For rectangles,
\[
\mu^{\otimes2}(A\times B)=\sum_{x\in A}\sum_{y\in B} m(x)m(y)
=\Big(\sum_{x\in A}m(x)\Big)\Big(\sum_{y\in B}m(y)\Big)=\mu(A)\mu(B).
\]
\end{proof}

\section{The axioms}\label{sec:axioms}

\begin{definition}[Pre-structural datum]\label{def:structural-datum}
A \emph{pre-structural datum} is a tuple
\[
\mathcal D=(X,\allowbreak \mathcal A,\allowbreak \mu,\allowbreak \mu^{\otimes2},\allowbreak R,\allowbreak I,\allowbreak \Pi_R,\allowbreak G,\allowbreak E_0,\allowbreak \eta)
\]
such that:
\begin{itemize}[nosep]
  \item $X$ is a nonempty set and $\mathcal A\subseteq\mathcal P(X)$ is an algebra;
  \item $\mu:\mathcal A\to[0,\infty)$ is a finitely additive measure;
  \item $\mu^{\otimes2}:\mathcal A\otimes\mathcal A\to[0,\infty)$ is a finitely additive product charge satisfying the rectangle rule
  \[
  \mu^{\otimes2}(B_1\times B_2)=\mu(B_1)\mu(B_2)\qquad(B_1,B_2\in\mathcal A);
  \]
  \item $R,I\in\mathcal A$ are disjoint and $R\cup I\in\mathcal A$;
  \item $\Pi_R:X\to R$ is a map;
  \item $G\subseteq X\times X$ lies in $\mathcal A\otimes\mathcal A$;
  \item $E_0\in(0,\infty)$ and $\eta\in[0,1]$ are scalars.
\end{itemize}
The pre-structural datum specifies only the raw signature of a structural model; the structural constraints are imposed separately by Axioms~I--III below.
\end{definition}

\paragraph{Axiom I (projection).}
$\Pi_R|_R=\mathrm{id}_R$; since $\Pi_R:X\to R$, this already forces the idempotence $\Pi_R\circ\Pi_R=\Pi_R$.
If $B\in\mathcal A\cap\mathcal P(R)$, then $\Pi_R^{-1}(B)\in\mathcal A$.
(The disjointness $R\cap I=\varnothing$ and the measurability $R\cup I\in\mathcal A$ are already imposed by the pre-structural datum.)

\paragraph{Axiom II (closed idempotent relation).}
$G$ is reflexive, symmetric, and idempotent: $G\circ G=G$.

\paragraph{Axiom III (Coupling).}
$\mu(R)+\mu(I)=E_0>0$.
If $B\subseteq R$ is measurable then $\mu(\Pi_R^{-1}(B))=\mu(B)$.
For all $B\in\mathcal A$,
\[
\mu^{\otimes2}((B\times X)\cap G)
= \mu(B)+\eta\,\mu^{\otimes2}((\Pi_R^{-1}(B)\times X)\cap G).
\]

\begin{remark}\label{rem:preimage}
Since $\Pi_R$ maps into $R$, one has $\Pi_R^{-1}(B)=\Pi_R^{-1}(B\cap R)$ for all $B\in\mathcal A$.
Thus $\Pi_R^{-1}(B)$ is measurable whenever $B\cap R$ is, by Axiom~I.
\end{remark}

\begin{remark}[Subclause~III(b) is the engine of collapse]\label{rem:IIIb-engine}
The projection--measure invariance subclause $\mu(\Pi_R^{-1}(B))=\mu(B)$ may look like a mild compatibility requirement, but it is the source of all the rigidity below. Already its instance at $B=R$ forces the measure onto $R$ (Theorem~\ref{thm:global-collapse}); the coupling law then propagates this collapse into the two--point load (Corollary~\ref{cor:rigidity}). The remaining subclauses fix the scale ($E_0$) and the propagation rate ($\eta$).
\end{remark}

\begin{definition}[Admissible structural model]\label{def:admissible}
A pre-structural datum
\[
\mathcal D=(X,\mathcal A,\mu,\mu^{\otimes2},R,I,\Pi_R,G,E_0,\eta)
\]
is called an \emph{admissible structural model} (or simply a \emph{structural model}) if it satisfies Axioms~I, II, and~III. When we refer to $\mathcal M$ as a \emph{structural model} without further qualification, admissibility is understood unless explicitly stated otherwise. The objects of the categories $\mathbf{Struct}_\eta$ and $\mathbf{Struct}^{\mathrm{id}}_\eta$ introduced in Section~\ref{sec:category} are admissible structural models by convention.
\end{definition}

\section{Collapse and rigidity}\label{sec:collapse}

This section contains the main structural results. After recording elementary consequences of Axioms~I and~II, we prove the global support collapse theorem, derive $\eta<1$ and the feasibility bound, and then obtain universal load rigidity by an elementary self--substitution. The fixed--point reformulation of the coupling law is placed last, as a corollary of collapse rather than as its engine.

\subsection{Elementary consequences}

\begin{proposition}[Elementary consequences of Axioms I and II]\label{prop:elementary}
From Axioms~I and~II one has immediately:
\begin{enumerate}[label=(\roman*),nosep]
\item (\emph{retraction fixed points}) $\Pi_R(x)=x$ if and only if $x\in R$;
\item (\emph{iterated projection}) $\Pi_R^n(x)=\Pi_R(x)$ for every integer $n\ge 1$; in particular $\Pi_R^n(x)\in R$;
\item (\emph{$G$ is an equivalence relation}) from reflexivity, symmetry, and idempotence $G\circ G=G$, transitivity follows: if $(x,y),(y,z)\in G$ then $(x,z)\in G\circ G=G$.
\end{enumerate}
\end{proposition}

\begin{proof}
(i) follows from $\Pi_R(X)\subseteq R$ and $\Pi_R|_R=\mathrm{id}_R$ (Axiom~I). (ii) is a consequence of idempotence $\Pi_R\circ\Pi_R=\Pi_R$: for $n=1$ this is definitional, and for $n\ge 2$ we have $\Pi_R^n(x)=\Pi_R(\Pi_R^{n-1}(x))=\Pi_R(x)$ by induction. (iii) is immediate from Axiom~II as stated.
\end{proof}

\begin{proposition}[Preimage idempotence]\label{prop:preimage-idempotence}
Under Axiom~I, the preimage map is idempotent: $\Pi_R^{-1}(\Pi_R^{-1}(B))=\Pi_R^{-1}(B)$ for all $B\in\mathcal A$.
\end{proposition}

\begin{proof}
For any $x\in X$, $x\in\Pi_R^{-1}(\Pi_R^{-1}(B))$ iff $\Pi_R(x)\in\Pi_R^{-1}(B)$ iff $\Pi_R(\Pi_R(x))\in B$ iff $\Pi_R(x)\in B$ (by idempotence of $\Pi_R$, Proposition~\ref{prop:elementary}(ii)) iff $x\in\Pi_R^{-1}(B)$.
\end{proof}

\subsection{Global support collapse}

The following is the central theorem of the paper. It says that the representative sector $R$ carries all of the $\mu$--mass, with no partition hypothesis: the conclusion holds for \emph{every} admissible model.

\begin{theorem}[Global support collapse]\label{thm:global-collapse}
Let $\mathcal M=(X,\mathcal A,\mu,\mu^{\otimes2},R,I,\Pi_R,G,E_0,\eta)$ be an admissible structural model. Then
\[
\mu(X\setminus R)=0.
\]
Consequently $\mu(I)=0$ and $\mu(X)=\mu(R)=E_0$. Moreover, for every $B\in\mathcal A$,
\[
\mu(\Pi_R^{-1}(B))=\mu(B).
\]
\end{theorem}

\begin{proof}
Apply projection--measure invariance (Axiom~III(b)) with $B=R$. Since $\Pi_R:X\to R$, we have $\Pi_R^{-1}(R)=X$, hence
\[
\mu(X)=\mu(\Pi_R^{-1}(R))=\mu(R).
\]
Because $R\in\mathcal A$ and $\mathcal A$ is an algebra, $X\setminus R\in\mathcal A$, and $X=R\sqcup(X\setminus R)$. Finite additivity gives
\[
\mu(X)=\mu(R)+\mu(X\setminus R),
\]
so $\mu(X\setminus R)=0$. Since $I\cap R=\varnothing$ we have $I\subseteq X\setminus R$, whence $\mu(I)=0$; and by conservation (Axiom~III(a)) $\mu(X)=\mu(R)=\mu(R)+\mu(I)=E_0$.

For the final identity, fix $B\in\mathcal A$. By Remark~\ref{rem:preimage}, $\Pi_R^{-1}(B)=\Pi_R^{-1}(B\cap R)$, and $B\cap R\subseteq R$ is measurable, so Axiom~III(b) applies to $B\cap R$:
\[
\mu(\Pi_R^{-1}(B))=\mu(\Pi_R^{-1}(B\cap R))=\mu(B\cap R).
\]
On the other hand $B=(B\cap R)\sqcup(B\setminus R)$ with $B\setminus R\subseteq X\setminus R$, so $\mu(B\setminus R)=0$ and $\mu(B)=\mu(B\cap R)$. Combining, $\mu(\Pi_R^{-1}(B))=\mu(B)$.
\end{proof}

\begin{remark}[Comparison with the partition case]\label{rem:partition-comparison}
The elementary observation that, under a full partition $X=R\sqcup I$, invariance forces $\mu(I)=0$, is the special case obtained by noting $X\setminus R=I$. Theorem~\ref{thm:global-collapse} removes the partition hypothesis entirely: the conclusion $\mu(X\setminus R)=0$ uses only $R\in\mathcal A$ and finite additivity. The non--representative part $X\setminus R$ may be large and may carry rich relational structure through $G$, yet it is $\mu$--null.
\end{remark}

\begin{remark}[Interpretation: $\mu$ is a readable measure]\label{rem:mu-readable}
The collapse $\mu(X\setminus R)=0$ should not be read as the nonexistence of non--representative data. It says only that the chosen measure $\mu$ is supported on the representative sector $R$; the null extension $X\setminus R$ may still carry set--theoretic, relational, or protocol data that is simply invisible to $\mu$. Accordingly, $\mu$ is to be interpreted as a \emph{readable} (representative) measure rather than as a tally of raw internal content: collapse concentrates what is read, not what exists. Statements about admissible models are correspondingly statements about the readable measure: they impose no \emph{positive-$\mu$-mass} constraint beyond the structural requirements that Axioms~I and~II already place on $X$, $\Pi_R$, and $G$, which continue to constrain even the $\mu$--null sector.
\end{remark}

\begin{example}[Why invariance forbids mass on nontrivial fibers]\label{ex:fiber-mass}
Let $X=\{r_0,r_1,a,b\}$, $R=\{r_0,r_1\}$, $I=\{a,b\}$, and $\Pi_R|_R=\mathrm{id}_R$, $\Pi_R(a)=r_0$, $\Pi_R(b)=r_1$.
If $\mu(\{a\})>0$, then for $B=\{r_0\}$ one has $\Pi_R^{-1}(B)=\{r_0,a\}$, so $\mu(\Pi_R^{-1}(B))=\mu(\{r_0\})+\mu(\{a\})>\mu(B)$, violating invariance. Theorem~\ref{thm:global-collapse} is the global form of this obstruction.
\end{example}

\subsection{Automatic exclusion of the endpoint and a feasibility bound}

\begin{proposition}[The endpoint $\eta=1$ is excluded]\label{prop:eta-lt-one}
For every admissible structural model, $\eta<1$.
\end{proposition}

\begin{proof}
Apply the coupling law (Axiom~III(c)) with $B=X$. Since $(X\times X)\cap G=G$ and $\Pi_R^{-1}(X)=X$,
\[
\mu^{\otimes2}(G)=\mu(X)+\eta\,\mu^{\otimes2}(G),\qquad\text{i.e.}\qquad (1-\eta)\,\mu^{\otimes2}(G)=\mu(X).
\]
By the rectangle rule, $\mu^{\otimes2}(G)\le\mu^{\otimes2}(X\times X)=\mu(X)^2<\infty$ (finite because $\mu$ is $[0,\infty)$--valued); thus the subtraction is valid. Moreover $\mu(X)\ge\mu(R)+\mu(I)=E_0>0$ by Axiom~III(a) and finite additivity, so the right side is positive. If $\eta=1$ the left side is $0$ while the right side is $\mu(X)>0$, a contradiction. Hence $\eta<1$. (This argument uses neither Axiom~II nor the support collapse of Theorem~\ref{thm:global-collapse}: it needs only the $B=X$ coupling law, the rectangle bound, and positivity from Axiom~III(a).)
\end{proof}

\begin{proposition}[Finite--additive feasibility bound]\label{prop:feasibility-bound}
For every admissible structural model,
\[
E_0=\mu(X)\ge\frac{1}{1-\eta},\qquad\text{equivalently}\qquad \eta\le 1-\frac{1}{E_0}.
\]
In particular, if $\mu(X)=1$ then $\eta=0$. The bound $\mu(X)\ge(1-\eta)^{-1}$ uses only the $B=X$ coupling law and the rectangle bound; the identification $E_0=\mu(X)$ uses in addition the collapse $\mu(X)=E_0$ (Theorem~\ref{thm:global-collapse}).
\end{proposition}

\begin{proof}
From the proof of Proposition~\ref{prop:eta-lt-one}, $\mu^{\otimes2}(G)=\mu(X)/(1-\eta)$. Since $\mu^{\otimes2}(G)\le\mu(X)^2$, we get $\mu(X)/(1-\eta)\le\mu(X)^2$, i.e.\ $1/(1-\eta)\le\mu(X)=E_0$. If $\mu(X)=1$ then $1/(1-\eta)\le 1$, forcing $\eta\le 0$, hence $\eta=0$.
\end{proof}

\begin{remark}[$\sigma$--additivity is not the essential issue]\label{rem:fa-essential}
Proposition~\ref{prop:feasibility-bound} shows that the obstruction to $\eta\neq 0$ under probability normalization $\mu(X)=1$ already appears for finitely additive probability measures: it is a consequence of the $B=X$ instance of the coupling law together with the rectangle rule, and does not require $\sigma$--additivity. The essential condition for nontrivial $\eta\neq 0$ is therefore the non--probability normalization $E_0>1$, not any countable--additivity assumption. Section~\ref{sec:normalization} isolates the genuinely $\sigma$--additive content (product--measure uniqueness and countable block formulas).
\end{remark}

\subsection{Universal load rigidity}

We now show that the two--point load is completely determined by $\mu$ and $\eta$. The proof is an elementary self--substitution using only collapse (Theorem~\ref{thm:global-collapse}), preimage idempotence (Proposition~\ref{prop:preimage-idempotence}), and $\eta<1$; in particular it does not use the Banach fixed--point theorem of Section~\ref{sec:fixed-point}.

\begin{corollary}[Universal load rigidity]\label{cor:rigidity}
For every admissible structural model and every $B\in\mathcal A$,
\[
\mu^{\otimes2}\bigl((B\times X)\cap G\bigr)=\frac{\mu(B)}{1-\eta}.
\]
In particular, the two--point load is not an independent datum once $\mu$ and $\eta$ are fixed.
\end{corollary}

\begin{proof}
Write $\mathcal C(B):=\mu^{\otimes2}((B\times X)\cap G)$. The coupling law reads
\[
\mathcal C(B)=\mu(B)+\eta\,\mathcal C(\Pi_R^{-1}(B))\qquad(\forall B\in\mathcal A).\tag{$\ast$}
\]
Apply $(\ast)$ to $\Pi_R^{-1}(B)$ in place of $B$:
\[
\mathcal C(\Pi_R^{-1}(B))=\mu(\Pi_R^{-1}(B))+\eta\,\mathcal C\bigl(\Pi_R^{-1}(\Pi_R^{-1}(B))\bigr).
\]
By preimage idempotence (Proposition~\ref{prop:preimage-idempotence}), $\Pi_R^{-1}(\Pi_R^{-1}(B))=\Pi_R^{-1}(B)$; and by collapse (Theorem~\ref{thm:global-collapse}), $\mu(\Pi_R^{-1}(B))=\mu(B)$. Hence
\[
\mathcal C(\Pi_R^{-1}(B))=\mu(B)+\eta\,\mathcal C(\Pi_R^{-1}(B)),
\]
so $(1-\eta)\,\mathcal C(\Pi_R^{-1}(B))=\mu(B)$. Since $\eta<1$ (Proposition~\ref{prop:eta-lt-one}) and $\mathcal C(\Pi_R^{-1}(B))\le\mu(X)^2<\infty$ (rectangle bound), we may divide:
\[
\mathcal C(\Pi_R^{-1}(B))=\frac{\mu(B)}{1-\eta}.
\]
Substituting back into $(\ast)$,
\[
\mathcal C(B)=\mu(B)+\eta\cdot\frac{\mu(B)}{1-\eta}=\frac{\mu(B)}{1-\eta}.
\]
\end{proof}

\begin{remark}[Global decoupling is automatic]\label{rem:decoupling-automatic}
One might consider the weaker hypothesis that the load on $\Pi_R^{-1}(B)\setminus B$ vanishes for $B\subseteq R$. Under collapse this is automatic: $\Pi_R^{-1}(B)\setminus B\subseteq X\setminus R$ (Proposition~\ref{prop:elementary}(i)), and any measurable $S\subseteq(X\setminus R)\times X$ satisfies $\mu^{\otimes2}(S)\le\mu(X\setminus R)\mu(X)=0$. Thus the load is supported, up to a $\mu^{\otimes2}$--null set, on $R\times R$, and Corollary~\ref{cor:rigidity} is the exact closed form. This subsumes any ``observable reduction on $R$'' statement.
\end{remark}

\subsection{Fixed-point reformulation}\label{sec:fixed-point}

We record the operator--theoretic reformulation of the coupling law. In view of Corollary~\ref{cor:rigidity} this is no longer needed to determine the load on an admissible model; we include it because $T_\eta$ is meaningful for arbitrary bounded charges (not only those arising from an admissible model), and because it exhibits the closed form as the unique fixed point of a contraction.

Throughout this subsection assume Axiom~I and $\mu(X)\le M<\infty$; Axiom~II is not needed, since the contraction and the fixed--point identification use only the measurability and idempotence of $\Pi_R$ (Axiom~I). Let
\[
\mathcal B:=\{\,f:\mathcal A\to\mathbb{R}\ \mid\ f\ \text{finitely additive and bounded}\,\},\qquad
\|f\|_\infty:=\sup_{B\in\mathcal A}|f(B)|.
\]
Under this norm $\mathcal B$ is a Banach space (see e.g.\ \cite{Rao,AliprantisBorder}); the cone $\mathcal B_+$ of nonnegative $f$ is closed.

\begin{definition}[Projection--update operator]\label{def:T-eta}
For $\eta\in[0,1)$, define $T_\eta:\mathcal B\to\mathcal B$ by $(T_\eta f)(B):=\mu(B)+\eta\,f(\Pi_R^{-1}(B))$.
\end{definition}

That $T_\eta$ maps $\mathcal B$ into itself, and preserves $\mathcal B_+$, follows as in the standard contraction setup: $\Pi_R^{-1}(B)\in\mathcal A$ (Remark~\ref{rem:preimage}), $T_\eta f$ is finitely additive in $B$, and $|T_\eta f(B)|\le M+\eta\|f\|_\infty$.

\begin{theorem}[Coupling law as unique fixed point]\label{thm:coupling-fixed-point}
Assume Axiom~I and $\mu(X)\le M<\infty$. For $\eta\in[0,1)$ the equation $f=T_\eta f$ has a unique solution $f_*\in\mathcal B$, given by
\[
f_*(B)=\mu(B)+\frac{\eta}{1-\eta}\,\mu(\Pi_R^{-1}(B))\qquad(B\in\mathcal A),
\]
with $f_*\in\mathcal B_+$ and $\|T_\eta^n f_0-f_*\|_\infty\le\eta^n\|f_0-f_*\|_\infty$ for every $f_0\in\mathcal B$.
\end{theorem}

\begin{proof}
For $f_1,f_2\in\mathcal B$, $|(T_\eta f_1-T_\eta f_2)(B)|=\eta|f_1(\Pi_R^{-1}(B))-f_2(\Pi_R^{-1}(B))|\le\eta\|f_1-f_2\|_\infty$, so $T_\eta$ is a strict contraction; Banach's theorem gives a unique fixed point and the convergence rate. To identify it, use preimage idempotence (Proposition~\ref{prop:preimage-idempotence}):
\begin{align*}
(T_\eta f_*)(B)
&=\mu(B)+\eta\Bigl[\mu(\Pi_R^{-1}(B))+\tfrac{\eta}{1-\eta}\mu\bigl(\Pi_R^{-1}(\Pi_R^{-1}(B))\bigr)\Bigr]\\
&=\mu(B)+\Bigl[\eta+\tfrac{\eta^2}{1-\eta}\Bigr]\mu(\Pi_R^{-1}(B))
=\mu(B)+\tfrac{\eta}{1-\eta}\mu(\Pi_R^{-1}(B))=f_*(B).
\end{align*}
Nonnegativity is immediate.
\end{proof}

\begin{corollary}[Consistency with rigidity]\label{cor:fixed-point-consistency}
On an admissible model, collapse (Theorem~\ref{thm:global-collapse}) gives $\mu(\Pi_R^{-1}(B))=\mu(B)$, so the fixed point simplifies to $f_*(B)=\mu(B)/(1-\eta)$, in agreement with Corollary~\ref{cor:rigidity}. The load $\mathcal C(B)=\mu^{\otimes2}((B\times X)\cap G)$ is finitely additive and bounded, and by Axiom~III(c) satisfies $\mathcal C=T_\eta\mathcal C$; hence it equals the unique fixed point $f_*$, giving a second proof of Corollary~\ref{cor:rigidity} for completeness.
\end{corollary}

\begin{remark}[Neumann series]\label{rem:neumann}
Iterating $T_\eta$ and using preimage idempotence, $f_*(B)=\sum_{n\ge0}\eta^n\mu(\Pi_R^{-n}(B))=\mu(B)+\frac{\eta}{1-\eta}\mu(\Pi_R^{-1}(B))$, the geometric sum $\tfrac{\eta}{1-\eta}$ accumulating the iterated projection corrections.
\end{remark}

\begin{remark}[The coupling law is not an ad-hoc postulate]\label{rem:coupling-not-ad-hoc}
The fixed--point form clarifies in what precise sense the coupling law of Axiom~III is a natural compatibility requirement rather than an arbitrary choice, once one fixes the \emph{form} of the update. The operator $T_\eta$ encodes exactly two conditions on a load $f:\mathcal A\to[0,\infty)$:
\begin{itemize}[nosep]
\item \emph{baseline agreement with $\mu$}: at $\eta=0$ one has $T_0 f=\mu$, so $f_*=\mu$ and the load coincides with the base measure;
\item \emph{self--similar propagation along $\Pi_R$}: the correction to $\mu$ propagates under the preimage map $\Pi_R^{-1}$, scaled by $\eta$.
\end{itemize}
We stress the precise scope of this uniqueness. The fixed--point formulation does \emph{not} prove that the coupling law of Axiom~III is the only conceivable compatibility law among all possible update rules; it shows only that \emph{once the affine self--similar update $T_\eta$ is adopted} (the two conditions above), the associated load is uniquely determined as $f_*$ once $(\mathcal A,\mu,\Pi_R,\eta)$ are fixed (Theorem~\ref{thm:coupling-fixed-point}). The determination of the load on an admissible model does not, however, rest on this choice of operator: by Corollary~\ref{cor:rigidity} the load equals $\mu(B)/(1-\eta)$ \emph{unconditionally}, regardless of how the coupling law is presented --- the propagation term, after collapse, no longer distinguishing $B$ from $\Pi_R^{-1}(B)$. It is on that unconditional rigidity, rather than on the particular operator $T_\eta$, that the closed form ultimately rests.
\end{remark}

\section{Models, consistency, and component quantization}\label{sec:models}

We first establish consistency by explicit models, then prove the measurable component quantization theorem, record the global block classification in the identity--retraction case, and give a non--atomic continuous model on $[0,1]$ showing that quantization persists outside the discrete setting.

\subsection{Finite models}

\begin{theorem}\label{thm:finite-model}
There exists a finite structural model. In particular, the axiom set is satisfiable in ZFC.
\end{theorem}

\begin{proof}
Let $X$ be finite and nonempty with disjoint nonempty $R,I\subseteq X$, $R\cup I=X$, $\mathcal A=\mathcal P(X)$. Fix $r_0\in R$ and set $\Pi_R(x)=x$ for $x\in R$, $\Pi_R(x)=r_0$ for $x\in I$; this is idempotent with $\Pi_R|_R=\mathrm{id}_R$, so Axiom~I holds. Let $G=\Delta_X$, giving Axiom~II. Choose $m(x)\in\{0,1\}$ with $m\equiv0$ on $I$ and $\sum_{x\in R}m(x)>0$, and $\mu(B)=\sum_{x\in B}m(x)$; then $\mu(I)=0$, $\mu(R)>0$, and $E_0:=\mu(X)=\mu(R)>0$. Let $\mu^{\otimes2}$ be the canonical atomic product charge induced by $m$ (Lemma~\ref{lem:atomic-product-charge}). For $B\subseteq R$, $\Pi_R^{-1}(B)$ is $B$ or $B\cup I$, and $\mu(I)=0$ gives invariance. With $\eta=0$ and $m(x)^2=m(x)$,
\[
\mu^{\otimes2}((B\times X)\cap G)=\sum_{x\in B}m(x)^2=\mu(B),
\]
so the coupling law holds. All axioms hold.
\end{proof}

\subsection{Finite models with \texorpdfstring{$\eta\neq 0$}{eta neq 0}}

\begin{theorem}[Finite model with $\eta\neq 0$]\label{thm:eta-nonzero-finite}
Let $X=\{r_0,r_1,i\}$, $\mathcal A=\mathcal P(X)$, $R=\{r_0,r_1\}$, $I=\{i\}$, $\Pi_R|_R=\mathrm{id}_R$, $\Pi_R(i)=r_0$, $G=\Delta_X$. Fix $\eta\in(0,1)$ and set $m(i)=0$, $m(r_0)=m(r_1)=\tfrac{1}{1-\eta}$, $\mu(B)=\sum_{x\in B}m(x)$, with $\mu^{\otimes2}$ the canonical atomic product charge induced by $m$ (Lemma~\ref{lem:atomic-product-charge}). Then all axioms hold, with $E_0=\tfrac{2}{1-\eta}$.
\end{theorem}

\begin{proof}
Axioms~I,~II are immediate; conservation holds by the choice of $E_0$; $m(i)=0$ gives invariance. For the coupling law, $\mu^{\otimes2}((B\times X)\cap G)=\sum_{x\in B}m(x)^2$, and the contribution of $i$ vanishes, so it reduces to $(1-\eta)m(r)^2=m(r)$ for $r\in R$, which holds for $m(r)=\tfrac{1}{1-\eta}$.
\end{proof}

\begin{proposition}[Non-diagonal $\eta\neq0$ model]\label{prop:eta-nonzero-nondiag}
Let $X=\{r_0,r_1\}$, $\mathcal A=\mathcal P(X)$, $R=X$, $I=\varnothing$, $\Pi_R=\mathrm{id}_X$, $G=X\times X$, $\eta\in(0,1)$, $m(r_0)=m(r_1)=\tfrac{1}{2(1-\eta)}$, $\mu(B)=\sum_{x\in B}m(x)$, and $\mu^{\otimes2}$ the canonical atomic product charge induced by $m$ (Lemma~\ref{lem:atomic-product-charge}). Then all axioms hold with $E_0=\mu(X)=\tfrac{1}{1-\eta}$, and $G$ is non-diagonal.
\end{proposition}

\begin{proof}
Invariance is trivial. For $B\subseteq X$, $(B\times X)\cap G=B\times X$, so $\mu^{\otimes2}((B\times X)\cap G)=\mu(B)\mu(X)$, and the coupling law $\mu(B)\mu(X)=\mu(B)+\eta\mu(B)\mu(X)$ reduces (for $\mu(B)>0$) to $(1-\eta)\mu(X)=1$, true by construction.
\end{proof}

\subsection{Countable model}

\begin{proposition}\label{prop:countable-model}
There exists a structural model with $X$ countably infinite.
\end{proposition}

\begin{proof}
Take $X=\{r\}\sqcup Y$ with $Y$ countably infinite, $R=\{r\}$, $I=Y$, $\mathcal A=\mathcal P(X)$, and $\Pi_R\equiv r$ (so $\Pi_R$ is idempotent with $\Pi_R|_R=\mathrm{id}_R$, giving Axiom~I). Define point masses $m(r)=1$ and $m(y)=0$ for $y\in Y$, so $\mu(B)=\sum_{x\in B}m(x)$ satisfies $\mu(\{r\})=1$, $\mu(Y)=0$, and $\sum_{x\in X}m(x)=1<\infty$; let $\mu^{\otimes2}$ be the canonical atomic product charge induced by $m$ (Lemma~\ref{lem:atomic-product-charge}). Put $G=X\times X$, $E_0=1$, $\eta=0$.

Crucially, $G=X\times X$ is a single rectangle, hence $G\in\mathcal A\otimes\mathcal A$ without any countable--union argument; this is why we take the total relation rather than a diagonal (which, for infinite $Y$, need not lie in the finitely generated product algebra $\mathcal A\otimes\mathcal A$). Axiom~II holds since $X\times X$ is reflexive, symmetric, and idempotent. For Axiom~III: conservation gives $\mu(R)+\mu(I)=1+0=1=E_0$; invariance holds since $\Pi_R^{-1}(\{r\})=X$ has $\mu(X)=1=\mu(\{r\})$; and for the coupling law, $(B\times X)\cap G=B\times X$, so by the rectangle rule
\[
\mu^{\otimes2}((B\times X)\cap G)=\mu(B)\,\mu(X)=\mu(B)\qquad(\forall B\subseteq X),
\]
which is the $\eta=0$ coupling law. Thus $\mathcal M$ is a countable structural model. (Consistently with Theorem~\ref{thm:global-collapse}, $\mu(X\setminus R)=\mu(Y)=0$.)
\end{proof}

\subsection{Measurable component quantization}

The collapse theorem upgrades the quantization phenomenon from the identity--retraction case to \emph{every} admissible model.

\begin{theorem}[Measurable component quantization]\label{thm:component-quantization}
Let $\mathcal M$ be an admissible structural model and let $C\in\mathcal A$ be a measurable $G$--equivalence class. Then
\[
\mu(C)=0\qquad\text{or}\qquad\mu(C)=\frac{1}{1-\eta}.
\]
\end{theorem}

\begin{proof}
By Proposition~\ref{prop:elementary}(iii), $G$ is an equivalence relation. For a $G$--equivalence class $C$, every $G$--partner of a point of $C$ again lies in $C$, so
\[
(C\times X)\cap G=C\times C.
\]
By the rectangle rule, $\mu^{\otimes2}((C\times X)\cap G)=\mu(C)^2$. By Universal Load Rigidity (Corollary~\ref{cor:rigidity}), the same quantity equals $\mu(C)/(1-\eta)$. Hence $\mu(C)^2=\mu(C)/(1-\eta)$, giving $\mu(C)=0$ or $\mu(C)=(1-\eta)^{-1}$.
\end{proof}

\begin{remark}\label{rem:quantization-scope}
Theorem~\ref{thm:component-quantization} applies to any \emph{measurable} class. Not every $G$--equivalence class need lie in $\mathcal A$; the hypothesis $C\in\mathcal A$ is what makes $\mu(C)$ and the rectangle rule available. In the identity--retraction case with finitely or countably many classes, all classes are measurable and one obtains the full classification of Theorem~\ref{thm:pi-id-classification}.
\end{remark}

\begin{corollary}[Integrality of the positive--mass block count]\label{cor:integrality}
Let $\mathcal M$ be an admissible structural model. Suppose there are finitely many positive--mass measurable $G$--equivalence classes $C_1,\dots,C_N\in\mathcal A$, and that their union is $\mu$--conull (i.e.\ $\mu(X\setminus\bigsqcup_{j}C_j)=0$). Then
\[
N=(1-\eta)\,E_0,
\]
and in particular $(1-\eta)E_0\in\mathbb{Z}_{>0}$ (the count is positive since $E_0>0$ and $\eta<1$). Consequently, under this finite positive--block exhaustion hypothesis, the parameters $\eta$ and $E_0$ cannot be chosen independently: they are tied through the integer block count.
\end{corollary}

\begin{proof}
By Theorem~\ref{thm:component-quantization} each $C_j$ has $\mu(C_j)=(1-\eta)^{-1}$. The classes are disjoint and their union is $\mu$--conull, so finite additivity gives
\[
E_0=\mu(X)=\sum_{j=1}^N\mu(C_j)=\frac{N}{1-\eta},
\]
using $\mu(X)=E_0$ from Theorem~\ref{thm:global-collapse}. Hence $N=(1-\eta)E_0$, a positive integer because $N\ge1$.
\end{proof}

\begin{remark}[A necessary condition visible only after combining the axioms]\label{rem:integrality-meaning}
The integrality constraint $(1-\eta)E_0\in\mathbb{Z}_{>0}$ is a genuine \emph{a priori} restriction on the scalar data of an admissible model: given a target total mass $E_0$ and a propagation rate $\eta$, no admissible model with the above block structure exists unless $(1-\eta)E_0$ happens to be an integer. Neither subclause of Axiom~III sees this on its own --- conservation only fixes $E_0$, the coupling law only relates load to $\mu$ --- and it becomes visible only once collapse, rigidity, and quantization are combined. It is in this sense a counting consequence of the rigidity package rather than of any single axiom.
\end{remark}

\subsection{Global block classification in the identity--retraction case}

When $\Pi_R=\mathrm{id}_X$ and the classes are summable, the quantization assembles into a block formula.

\begin{theorem}[Classification when $\Pi_R=\mathrm{id}$]\label{thm:pi-id-classification}
Assume $R=X$, $I=\varnothing$, $\Pi_R=\mathrm{id}_X$, $\eta\in[0,1)$, and let $G$ be reflexive, symmetric, idempotent, with equivalence classes $(C_k)_{k\in K}$, so $G=\bigsqcup_k(C_k\times C_k)$. Let $\mu$ be finitely additive and $\mu^{\otimes2}$ a product charge (Definition~\ref{def:product-measure}). Assume that each class is measurable, $C_k\in\mathcal A$, and one of:
\begin{itemize}[nosep]
\item[\emph{(K--fin)}] $K$ is finite; or
\item[\emph{(K--ctbl)}] $K$ is countable, $\mathcal A$ is a $\sigma$--algebra, $\mu$ is $\sigma$--additive and $\sigma$--finite, and $\mu^{\otimes2}$ is the standard $\sigma$--additive product measure on the product $\sigma$--algebra $\mathcal A\,\overline{\otimes}\,\mathcal A$, in which the countable union $G=\bigsqcup_k C_k\times C_k$ lies; its restriction to the product algebra $\mathcal A\otimes\mathcal A$ is the product charge of Definition~\ref{def:product-measure}, and we use the same symbol $\mu^{\otimes2}$ for both.
\end{itemize}
Then for every measurable $B$ the \emph{block formula} holds,
\[
\mu^{\otimes2}((B\times X)\cap G)=\sum_{k\in K}\mu(B\cap C_k)\,\mu(C_k),
\]
and, consequently, the coupling law holds if and only if every positive--mass class satisfies $\mu(C_k)=(1-\eta)^{-1}$.
\end{theorem}

\begin{proof}
For measurable $B$, $(B\times X)\cap G=\bigsqcup_k((B\cap C_k)\times C_k)$, a disjoint union of measurable rectangles (using $C_k\in\mathcal A$). Under~(K--fin) the finite additivity of $\mu^{\otimes2}$, and under~(K--ctbl) its $\sigma$--additivity on $\mathcal A\,\overline{\otimes}\,\mathcal A$, together with the rectangle rule, yield the block formula. Now suppose the coupling law holds. Since $\Pi_R=\mathrm{id}_X$ it reads $(1-\eta)\mu^{\otimes2}((B\times X)\cap G)=\mu(B)$; taking $B=C_j$ (measurable) gives $(1-\eta)\mu(C_j)^2=\mu(C_j)$, hence $\mu(C_j)\in\{0,(1-\eta)^{-1}\}$, so every positive--mass class satisfies $\mu(C_j)=(1-\eta)^{-1}$. Conversely, suppose every positive--mass class satisfies $\mu(C_k)=(1-\eta)^{-1}$. For a zero--mass class, $B\cap C_k\subseteq C_k$ gives $\mu(B\cap C_k)=0$ by monotonicity, so such classes contribute nothing to the block sum; by the same finite or countable additivity over the partition $(C_k)$ one has $\sum_{k:\,\mu(C_k)>0}\mu(B\cap C_k)=\mu(B)$. Hence the block formula gives
\[
\mu^{\otimes2}((B\times X)\cap G)=\sum_{k:\,\mu(C_k)>0}\frac{\mu(B\cap C_k)}{1-\eta}=\frac{\mu(B)}{1-\eta},
\]
which is the coupling law.
\end{proof}

\begin{remark}[On the summability hypothesis]\label{rem:block-summability}
The block formula involves a sum over $K$ and can fail in the purely finitely additive setting when $K$ is infinite, since countable disjoint unions need not be countably additive for finitely additive $\mu^{\otimes2}$. The hypotheses~(K--fin) or~(K--ctbl) guarantee it. Note that the \emph{pointwise} quantization $\mu(C)\in\{0,(1-\eta)^{-1}\}$ of Theorem~\ref{thm:component-quantization} requires \emph{no} summability hypothesis: it follows from rigidity applied to the single measurable set $C$.
\end{remark}

\begin{corollary}[Total mass, finitely many positive classes]\label{cor:pi-id-total-mass}
Under Theorem~\ref{thm:pi-id-classification} with finitely many classes $C_1,\dots,C_m$, all of positive mass,
\[
\mu(X)=\frac{m}{1-\eta},\qquad \mu^{\otimes2}(G)=\frac{m}{(1-\eta)^2},\qquad
\mu^{\otimes2}((B\times X)\cap G)=\frac{\mu(B)}{1-\eta}.
\]
\end{corollary}

\begin{proof}
$\mu(C_j)=(1-\eta)^{-1}$ for each $j$; sum over the partition $X=\bigsqcup_j C_j$ and over $G=\bigsqcup_j(C_j\times C_j)$. The last identity is Corollary~\ref{cor:rigidity}.
\end{proof}

\subsection{A continuous Lebesgue model}\label{sec:continuous-model}

The models above are finite or atomic. To show that the axioms admit genuinely non--atomic, continuous models, and to illustrate how collapse and quantization act on a non--discrete measure, we give a model on $[0,1]$ built from Lebesgue measure. The collapse theorem does not forbid such models; rather, it pins down their structure precisely, as the construction makes explicit.

\begin{proposition}[Continuous model on the unit interval]\label{prop:continuous-model}
Fix an integer $m\ge1$ and $\eta\in[0,1)$. Let $X=[0,1]$ with the Borel $\sigma$--algebra $\mathcal A=\mathcal B([0,1])$, and let $\lambda$ be Lebesgue measure. Partition $[0,1]$ into $m$ Borel blocks $C_1,\dots,C_m$ of equal length $1/m$ (e.g.\ $C_k=[\tfrac{k-1}{m},\tfrac{k}{m})$, with $C_m$ closed at $1$). Define
\[
\mu:=\frac{m}{1-\eta}\,\lambda,\qquad
R:=X,\quad I:=\varnothing,\quad \Pi_R:=\mathrm{id}_X,\quad
G:=\bigsqcup_{k=1}^m (C_k\times C_k),
\]
let $\bar\mu^{\otimes2}$ be the standard product measure of $\mu$ on the product $\sigma$--algebra $\mathcal B([0,1])\,\overline{\otimes}\,\mathcal B([0,1])=\mathcal B([0,1]^2)$, and let $\mu^{\otimes2}$ denote its restriction to the product algebra $\mathcal A\otimes\mathcal A$ (the datum's product charge; note $G$, a finite union of rectangles, lies in $\mathcal A\otimes\mathcal A$), and set $E_0:=\mu(X)=\tfrac{m}{1-\eta}$. Then $\mathcal M=(X,\mathcal A,\mu,\mu^{\otimes2},R,I,\Pi_R,G,E_0,\eta)$ is an admissible structural model, $\mu$ is non--atomic, and each block satisfies the quantization value
\[
\mu(C_k)=\frac{1}{1-\eta}\qquad(k=1,\dots,m).
\]
\end{proposition}

\begin{proof}
Axiom~I holds trivially ($\Pi_R=\mathrm{id}_X$), and $G$ is reflexive, symmetric, and idempotent as a disjoint union of full blocks $C_k\times C_k$, giving Axiom~II; here $G\in\mathcal A\otimes\mathcal A$ is a finite union of measurable rectangles. For Axiom~III, conservation gives $\mu(R)+\mu(I)=\mu(X)=\tfrac{m}{1-\eta}=E_0>0$, and invariance is trivial since $\Pi_R=\mathrm{id}_X$. Each $C_k$ has $\mu(C_k)=\tfrac{m}{1-\eta}\,\lambda(C_k)=\tfrac{m}{1-\eta}\cdot\tfrac1m=\tfrac{1}{1-\eta}$. Since $\Pi_R=\mathrm{id}_X$ and the classes are the $C_k$ (finite in number, hence (K--fin)), Theorem~\ref{thm:pi-id-classification} applies: every positive--mass class has mass $(1-\eta)^{-1}$, so the coupling law holds. Finally $\mu=\tfrac{m}{1-\eta}\lambda$ is non--atomic because $\lambda$ is.
\end{proof}

\begin{remark}[Collapse and quantization in the continuous setting]\label{rem:continuous-quantization}
This example shows that the rigidity theorems are not artifacts of discreteness. Collapse (Theorem~\ref{thm:global-collapse}) is consistent with a non--atomic $\mu$: here $R=X$, so $X\setminus R=\varnothing$ and collapse is vacuous, while $\mu$ itself remains continuous. Component quantization (Theorem~\ref{thm:component-quantization}) does bite, and in a structurally informative way: although $\mu$ is non--atomic, the \emph{measurable $G$--equivalence classes} cannot have arbitrary mass\,---\,each positive--mass class is forced to mass exactly $(1-\eta)^{-1}$. Consequently the number of positive--mass classes is rigidly bounded:
\[
\#\{k:\mu(C_k)>0\}=(1-\eta)\,\mu(X)=(1-\eta)E_0,
\]
which must be a nonnegative integer, and a positive one ($\in\mathbb{Z}_{>0}$) whenever any positive--mass class is present\,---\,consistently with Corollary~\ref{cor:integrality}. Thus, even on a continuum, the relation $G$ may partition the space into only finitely many positive--mass blocks (here $m$), each of identical mass, with the remainder $\mu$--null. The continuum supplies the points; the axioms quantize how $G$ may bundle them.
\end{remark}

\begin{remark}[A non--identity continuous variant]\label{rem:continuous-nonid}
One can also exhibit collapse non--vacuously in the continuous setting. Take $X=[0,2]$, $\mathcal A=\mathcal B([0,2])$, $R=[0,1]$, $I=(1,2]$, and
\[
\Pi_R(x)=x\ \ (x\in[0,1]),\qquad \Pi_R(x)=0\ \ (x\in(1,2]),
\]
which is idempotent with $\Pi_R|_R=\mathrm{id}_R$. Let
\[
\mu(B)=\frac{m}{1-\eta}\,\lambda(B\cap[0,1]),
\]
so $\mu$ is carried by $R$ and $\mu(I)=0$. Let the blocks $C_1,\dots,C_m\subseteq[0,1]$ be as in Proposition~\ref{prop:continuous-model}, and set
\[
G:=\Big(\bigsqcup_{k=1}^m (C_k\times C_k)\Big)\ \sqcup\ (I\times I).
\]
Take the same $m\ge1$ and $\eta\in[0,1)$ as in Proposition~\ref{prop:continuous-model}, let $\mu^{\otimes2}$ be the product measure of $\mu$ restricted to $\mathcal A\otimes\mathcal A$, and set $E_0:=\mu(X)=\tfrac{m}{1-\eta}$. Then $G\in\mathcal A\otimes\mathcal A$ as a finite union of rectangles, and $G$ is reflexive, symmetric, and idempotent, with equivalence classes $C_1,\dots,C_m,I$. Since $\mu(I)=0$, the class $I$ contributes no load, and the coupling law reduces to that of the representative core on $R$. Thus
\[
\mu(X\setminus R)=\mu(I)=0
\]
realizes collapse with a nontrivial hidden sector $I=(1,2]$ of positive set--theoretic size but zero $\mu$--mass, exactly as Theorem~\ref{thm:global-collapse} predicts; the core $\mathcal M|_R$ recovers Proposition~\ref{prop:continuous-model}.
\end{remark}


\section{Independence of the axioms}\label{sec:independence}

The collapse and rigidity results of Section~\ref{sec:collapse} hold for admissible models, i.e.\ when all three axioms are present. We now show that no axiom is redundant: the three axioms, and the three subclauses of Axiom~III, are mutually independent. This makes precise the division of labor noted in the introduction: support collapse is produced by subclause~III(b); the endpoint exclusion and load rigidity then use the coupling law~III(c) (with the idempotence of $\Pi_R$ from Axiom~I for rigidity) but not Axiom~II, while component quantization is where Axiom~II becomes essential; none of these is subsumed by the others.

Independence is established model-theoretically: for each $\mathsf X$ we exhibit a structure over the same signature satisfying the other axioms but failing $\mathsf X$.

\subsection{Separating models for Axioms I, II, III}

\begin{proposition}[Independence of Axiom I]\label{prop:indep-I}
There is a pre-structural datum satisfying Axioms~II,~III but failing Axiom~I.
\end{proposition}

\begin{proof}
Let $X=\{a,b,c\}$, $\mathcal A=\mathcal P(X)$, $R=\{a,b\}$, $I=\{c\}$, and $\Pi_R(a)=b$, $\Pi_R(b)=a$, $\Pi_R(c)=a$. Then $\Pi_R|_R\ne\mathrm{id}_R$ and $\Pi_R\circ\Pi_R\ne\Pi_R$, so Axiom~I fails. Set $G=\Delta_X$ (Axiom~II). Let $m(a)=m(b)=1$, $m(c)=0$, with $\mu(B)=\sum_{x\in B}m(x)$ and $\mu^{\otimes2}$ the canonical atomic product charge induced by $m$ (Lemma~\ref{lem:atomic-product-charge}); set $\eta=0$, $E_0=2$. Conservation holds. Invariance: $\Pi_R^{-1}(\{a\})=\{b,c\}$ has mass $1=\mu(\{a\})$, $\Pi_R^{-1}(\{b\})=\{a\}$ has mass $1$, and $\Pi_R^{-1}(R)=X$ has mass $2$; so $\mu(\Pi_R^{-1}(B))=\mu(B)$ for $B\subseteq R$. The coupling law with $\eta=0$ and $G=\Delta_X$ reduces to $\sum_{x\in B}m(x)^2=\mu(B)$, which holds. Thus Axiom~III holds while Axiom~I fails.
\end{proof}

\begin{proposition}[Independence of Axiom II]\label{prop:indep-II}
There is a pre-structural datum satisfying Axioms~I,~III but failing Axiom~II.
\end{proposition}

\begin{proof}
Let $X=\{a,b,c\}$, $\mathcal A=\mathcal P(X)$, $R=X$, $I=\varnothing$, $\Pi_R=\mathrm{id}_X$ (Axiom~I). Let
\[
G=\{(a,a),(b,b),(c,c),(a,b),(b,a),(b,c),(c,b)\}.
\]
$G$ is reflexive and symmetric but not idempotent: $(a,b),(b,c)\in G$ yet $(a,c)\notin G$, so $G\circ G\supsetneq G$ and Axiom~II fails. Let $m(a)=1$, $m(b)=0$, $m(c)=1$, with $\mu(B)=\sum_{x\in B}m(x)$ and $\mu^{\otimes2}$ the canonical atomic product charge induced by $m$ (Lemma~\ref{lem:atomic-product-charge}); set $\eta=0$, $E_0=2$. Invariance is trivial. The only $G$--pairs with nonzero $m$--weight are $(a,a)$ and $(c,c)$, so $\mu^{\otimes2}((B\times X)\cap G)=\mu(B)$ for all $B$, and the coupling law holds. Thus Axiom~III holds while Axiom~II fails.
\end{proof}

\begin{remark}\label{rem:II-invisibility}
In this datum the failure $G\circ G\supsetneq G$ occurs at $(a,c)$, which carries weight $m(a)m(c)=1$ but lies outside $G$. The witness is set-theoretic, not measure-theoretic: Axiom~II is a structural statement, visible at the level of $G$ regardless of whether the coupling law detects it.
\end{remark}

\begin{proposition}[Independence of Axiom III]\label{prop:indep-III}
There is a pre-structural datum satisfying Axioms~I,~II but failing Axiom~III.
\end{proposition}

\begin{proof}
Let $X=\{r,i\}$, $\mathcal A=\mathcal P(X)$, $R=\{r\}$, $I=\{i\}$, $\Pi_R(r)=\Pi_R(i)=r$ (Axiom~I), $G=\Delta_X$ (Axiom~II). Let $m(r)=m(i)=1$, with $\mu(B)=\sum_{x\in B}m(x)$ and $\mu^{\otimes2}$ the canonical atomic product charge induced by $m$ (Lemma~\ref{lem:atomic-product-charge}); set $\eta=0$, $E_0=2$ (conservation holds). Invariance fails: $B=\{r\}$ gives $\Pi_R^{-1}(\{r\})=X$, so $\mu(\Pi_R^{-1}(\{r\}))=2\ne1=\mu(\{r\})$. Hence Axiom~III fails. (Indeed, this datum violates collapse: $\mu(X\setminus R)=\mu(\{i\})=1\ne0$, exhibiting the necessity of III(b).)
\end{proof}

\subsection{Internal independence of Axiom III}\label{sec:indep-III-internal}

Axiom~III has subclauses (a) conservation, (b) invariance, (c) coupling law.

\begin{proposition}[Internal independence]\label{prop:indep-III-internal}
Subclauses (a), (b), (c) are mutually independent relative to Axioms~I,~II.
\end{proposition}

\begin{proof}
\emph{(a) independent.} $X=\{r\}$, $R=\{r\}$, $I=\varnothing$, $\Pi_R=\mathrm{id}$, $G=\Delta_X$, $\mu\equiv0$, and $\mu^{\otimes2}\equiv0$ the corresponding zero product charge, $\eta=0$, and choose $E_0=1$ (the signature requires $E_0\in(0,\infty)$). Then (b) is trivial and (c) holds ($0=0$), but $\mu(R)+\mu(I)=0\ne1=E_0$, so conservation (a) fails.

\emph{(b) independent.} The datum $\mathcal D_{\neg\mathrm{III}}$ of Proposition~\ref{prop:indep-III} satisfies (a) ($E_0=2$) and (c) (with $\eta=0$ and $G=\Delta_X$: $\mu^{\otimes2}((\{x\}\times X)\cap\Delta_X)=m(x)^2=m(x)=\mu(\{x\})$ on singletons, hence (c) for all $B$ by finite additivity), but (b) fails.

\emph{(c) independent.} $X=\{a,b\}$, $R=X$, $I=\varnothing$, $\Pi_R=\mathrm{id}$, $G=X\times X$, with $m(a)=1$, $m(b)=0$, $\mu(B)=\sum_{x\in B}m(x)$, and $\mu^{\otimes2}$ the canonical atomic product charge induced by $m$ (Lemma~\ref{lem:atomic-product-charge}), $\eta=\tfrac12$: (a) holds ($E_0=1$), (b) trivial. But for $B=\{a\}$, $\mu^{\otimes2}((\{a\}\times X)\cap G)=m(a)(m(a)+m(b))=1$ while $\mu(B)+\eta\,\mu^{\otimes2}((\Pi_R^{-1}(B)\times X)\cap G)=1+\tfrac12\cdot1=\tfrac32$, so (c) fails.
\end{proof}

\subsection{Main independence theorem}\label{sec:indep-main}

\begin{theorem}[Independence]\label{thm:independence-main}
The axioms~I,~II,~III are mutually independent, and the subclauses (a),(b),(c) of Axiom~III are mutually independent relative to Axioms~I,~II.
\end{theorem}

\begin{proof}
Combine Propositions~\ref{prop:indep-I},~\ref{prop:indep-II},~\ref{prop:indep-III},~\ref{prop:indep-III-internal}.
\end{proof}

\begin{corollary}[Minimality]\label{cor:minimality}
Removing any one axiom, or any one subclause of Axiom~III, yields a strictly weaker system. In particular, support collapse (Theorem~\ref{thm:global-collapse}) already fails without subclause~III(b) (Proposition~\ref{prop:indep-III}); load rigidity (Corollary~\ref{cor:rigidity}) additionally requires III(b) (through collapse), the coupling law~III(c), and the idempotence of $\Pi_R$ (Axiom~I), but \emph{not} Axiom~II; and Axiom~II becomes essential only for component quantization (Theorem~\ref{thm:component-quantization}), where it makes $G$ an equivalence relation.
\end{corollary}

\begin{proof}
Immediate from Theorem~\ref{thm:independence-main}; the separating models witness each claim.
\end{proof}

\begin{remark}[Independence, yet collapse]\label{rem:indep-yet-collapse}
The conjunction of Theorem~\ref{thm:independence-main} and the results of Section~\ref{sec:collapse} is the conceptual core of the paper. The three axioms are independent (none is redundant), yet \emph{together} they force the rigidity package. The layers should be kept distinct (see Table~\ref{tab:dependencies}): support collapse $\mu(X\setminus R)=0$ already follows from subclause~III(b) (Theorem~\ref{thm:global-collapse}); the endpoint exclusion and load rigidity require, in addition, the coupling law~III(c) (and, for rigidity, collapse via~III(b) together with the idempotence of $\Pi_R$ from Axiom~I), but \emph{not} Axiom~II; and only component quantization invokes the idempotent structure of~II, through which $G$ becomes an equivalence relation. The independence results certify that this layered rigidity is produced by genuine interaction, not by a hidden subsumption of one axiom by another.
\end{remark}

\begin{table}[h]
\centering
\small
\renewcommand{\arraystretch}{1.25}
\setlength{\tabcolsep}{4pt}
\begin{tabular}{lll}
\hline
\textbf{Consequence} & \textbf{Essentially used} & \textbf{Not used} \\
\hline
Support collapse $\mu(X\setminus R){=}0$ (Thm~\ref{thm:global-collapse}) & III(b) $[B{=}R]$, finite additivity & I, II, III(c) \\
Endpoint $\eta<1$ (Prop~\ref{prop:eta-lt-one}) & III(c) $[B{=}X]$, III(a), rectangle bound & I, II, III(b) \\
Load rigidity (Cor~\ref{cor:rigidity}) & III(b), III(c), Axiom~I, $\eta<1$ & II \\
Quantization (Thm~\ref{thm:component-quantization}) & load rigidity, Axiom~II & --- \\
Integrality (Cor~\ref{cor:integrality}) & quantization, finite $\mu$--conull blocks & --- \\
\hline
\end{tabular}
\caption{Division of labor among the axioms. Each consequence is paired with the axioms and subclauses essential to its proof; the last column records ingredients the proof does \emph{not} require. The idempotent--relation axiom~II ($G\circ G=G$) enters only from component quantization onward, where it makes $G$ an equivalence relation; support collapse, endpoint exclusion, and load rigidity do not use it. The support--collapse row refers to the conclusion $\mu(X\setminus R)=0$; the auxiliary identity $\mu(\Pi_R^{-1}(B))=\mu(B)$ for all $B\in\mathcal A$ proved in the same theorem additionally invokes the preimage--measurability clause of Axiom~I, and it is this identity (not the bare conclusion $\mu(X\setminus R)=0$) that the load--rigidity row bills as ``Axiom~I''. Throughout, ``Axiom~I'' enters only through preimage idempotence $\Pi_R^{-1}\Pi_R^{-1}=\Pi_R^{-1}$, itself a consequence of the retraction property $\Pi_R|_R=\mathrm{id}_R$ rather than a separately imposed idempotence.}
\label{tab:dependencies}
\end{table}


\section{Normalization constraints and the role of \texorpdfstring{$\sigma$}{sigma}--additivity}\label{sec:normalization}

Proposition~\ref{prop:feasibility-bound} already gave the feasibility bound $E_0\ge(1-\eta)^{-1}$ and the implication $\mu(X)=1\Rightarrow\eta=0$ using only finite additivity. This section isolates what, if anything, $\sigma$--additivity adds. The short answer: it is needed only to guarantee a canonical product measure and to license countable block formulas; the normalization obstruction itself is finitely additive.

\subsection{The global identity}
\begin{proposition}[Global identity at $B=X$]\label{prop:global-constraint}
For any pre-structural datum with $\mu(X)<\infty$ satisfying the coupling law at $B=X$,
\[
(1-\eta)\,\mu^{\otimes2}(G)=\mu(X).
\]
If moreover $\mu(X)\in(0,\infty)$, then necessarily $\eta<1$ (else the left side vanishes while the right is positive), and hence
\[
\mu^{\otimes2}(G)=\frac{\mu(X)}{1-\eta}.
\]
\end{proposition}

\begin{proof}
$(X\times X)\cap G=G$ and $\Pi_R^{-1}(X)=X$, so the coupling law at $B=X$ reads $\mu^{\otimes2}(G)=\mu(X)+\eta\mu^{\otimes2}(G)$, i.e.\ $(1-\eta)\mu^{\otimes2}(G)=\mu(X)$. If $\mu(X)>0$ and $\eta=1$ this would force $0=\mu(X)>0$, impossible; so $\eta<1$ and division is valid.
\end{proof}

\begin{corollary}[Probability normalization forces $\eta=0$]\label{cor:probability-forces-eta0}
If $\mu(X)=1$ (finitely additive probability), then $\eta=0$ and $\mu^{\otimes2}(G)=1$.
\end{corollary}

\begin{proof}
$\mu^{\otimes2}(G)=1/(1-\eta)\ge1$, but $\mu^{\otimes2}(G)\le\mu^{\otimes2}(X\times X)=1$; hence $\mu^{\otimes2}(G)=1$ and $\eta=0$.
\end{proof}

\subsection{\texorpdfstring{Where $\sigma$--additivity is genuinely used}{Where sigma-additivity is genuinely used}}
The only places countable additivity is essential are:
\begin{itemize}[nosep]
\item \emph{Product--measure uniqueness.} Under $\sigma$--additivity and $\sigma$--finiteness the product measure $\mu^{\otimes2}$ on $(X\times X,\Sigma\,\overline{\otimes}\,\Sigma)$ exists and is uniquely determined by the rectangle rule (see e.g.\ \cite{Halmos,Bogachev}); in the purely finitely additive setting $\mu^{\otimes2}$ must be posited as model data.
\item \emph{Countable block formulas.} The block formula of Theorem~\ref{thm:pi-id-classification} under~(K--ctbl) requires $\sigma$--additivity to sum over countably many classes; the finite case~(K--fin) does not.
\end{itemize}
Neither of these affects the normalization obstruction, which is the content of Proposition~\ref{prop:feasibility-bound} and Corollary~\ref{cor:probability-forces-eta0}.

\begin{remark}[Corrected normalization statement]\label{rem:normalization-corrected}
The precise statement is therefore: \emph{nontrivial $\eta\neq 0$ families require non--probability total mass, $E_0=\mu(X)>1$; this obstruction already appears for finitely additive probability measures, and $\sigma$--additivity is not the essential issue.} We emphasize that infinite total mass is not an admissible regime: by definition $\mu$ is $[0,\infty)$--valued, and even granting an extended--valued reading, Axiom~III(b) at $B=R$ gives $\mu(X)=\mu(R)$ while Axiom~III(a) gives $\mu(R)\le E_0<\infty$, so every admissible model satisfies $\mu(X)=E_0<\infty$ automatically. Large total mass can therefore be approached only through \emph{sequences} of admissible models with $E_{0,n}\to\infty$, not through a single infinite--mass model; see the further directions in Section~\ref{sec:conclusion}.
\end{remark}


\section{Categorical structure and null-extension factorization}\label{sec:category}

We organize structural models into a category (for categorical background see \cite{MacLane,Awodey,Borceux,Riehl}), and show that the collapse theorem makes the passage to the identity--retraction core a clean null--extension factorization, valid on \emph{all} of $\mathbf{Struct}_\eta$ without fiber hypotheses.

\subsection{Morphisms and the category}

We fix $\eta\in[0,1)$ and treat $E_0\in(0,\infty)$ as object data; $\mathbf{Struct}_\eta$ has admissible structural models as objects. As Remark~\ref{rem:E0-morphism} below shows, every morphism preserves $E_0$, so this object datum is in fact a morphism invariant.

\begin{definition}[Morphism]\label{def:morphism}
Let $\mathcal M,\mathcal M'$ be admissible models sharing $\eta$. A \emph{morphism} $\phi:\mathcal M\to\mathcal M'$ is a map $\phi:X\to X'$ with:
\begin{enumerate}[label=\emph{(M\arabic*)},nosep,leftmargin=2.5em]
\item \emph{Measurability}: $\phi^{-1}(B')\in\mathcal A$ and $(\phi\times\phi)^{-1}(S')\in\mathcal A\otimes\mathcal A$ for $B'\in\mathcal A'$, $S'\in\mathcal A'\otimes\mathcal A'$.
\item \emph{Projection commutativity}: $\phi\circ\Pi_R=\Pi_{R'}\circ\phi$.
\item \emph{Relation preservation}: $(\phi\times\phi)(G)\subseteq G'$.
\item \emph{Measure--preservation}: for every $B'\in\mathcal A'$ and $S'\in\mathcal A'\otimes\mathcal A'$,
\[
\mu(\phi^{-1}(B'))=\mu'(B'),\qquad \mu^{\otimes2}((\phi\times\phi)^{-1}(S'))=\mu'^{\otimes2}(S').
\]
\end{enumerate}
\end{definition}

\begin{remark}[$E_0$ under morphisms]\label{rem:E0-morphism}
Applying (M4) to $B'=X'$, and noting $\phi^{-1}(X')=X$, gives $\mu(X)=\mu'(X')$. Since both objects are admissible, Theorem~\ref{thm:global-collapse} gives $\mu(X)=E_0$ and $\mu'(X')=E_0'$. Hence \emph{every} morphism in $\mathbf{Struct}_\eta$ preserves the scalar,
\[
E_0=E_0',
\]
with no surjectivity hypothesis. Consequently $\mathbf{Struct}_\eta$ decomposes as a disjoint union of full subcategories indexed by the value of $E_0\in(0,\infty)$: there are no morphisms between objects of different total mass.
\end{remark}

\begin{remark}[Automatic consequences]\label{rem:morphism-consequences}
$\phi(R)=\phi(\Pi_R(X))=\Pi_{R'}(\phi(X))\subseteq R'$ by~(M2) and Axiom~I; and $\phi_*\mu^{\otimes2}=\mu'^{\otimes2}$ on rectangles by~(M4) and the rectangle rule.
\end{remark}

\begin{proposition}[Category]\label{prop:category}
Morphisms compose and identities exist, so admissible models with common $\eta$ form a category $\mathbf{Struct}_\eta$.
\end{proposition}

\begin{proof}
(M1)--(M4) are preserved under composition: measurability by $(\psi\circ\phi)^{-1}=\phi^{-1}\circ\psi^{-1}$; (M2) by associativity; (M3) by $(\psi\times\psi)(\phi\times\phi)(G)\subseteq(\psi\times\psi)(G')\subseteq G''$; and (M4) by
\[
\mu((\psi\circ\phi)^{-1}(B''))=\mu'(\psi^{-1}(B''))=\mu''(B'')
\]
for the measure part, together with the product part, for $S''\in\mathcal A''\otimes\mathcal A''$,
\[
\mu^{\otimes2}\bigl(((\psi\circ\phi)\times(\psi\circ\phi))^{-1}(S'')\bigr)
=\mu'^{\otimes2}\bigl((\psi\times\psi)^{-1}(S'')\bigr)
=\mu''^{\otimes2}(S''),
\]
using $((\psi\circ\phi)\times(\psi\circ\phi))^{-1}=(\phi\times\phi)^{-1}\circ(\psi\times\psi)^{-1}$. Identity is clear.
\end{proof}

\subsection{Transport of the coupling law}

Morphisms preserve not only the listed data but also the coupling law. By universal load rigidity (Corollary~\ref{cor:rigidity}) the load is already determined on each object by $(\mu,\eta)$; the following records how the closed form transports along a morphism, and shows that the excess of the pulled--back relation is load--null.

\begin{proposition}[Transport of the coupling law]\label{prop:coupling-transport}
Let $\phi:\mathcal M\to\mathcal M'$ be a morphism in $\mathbf{Struct}_\eta$, and write $P_\phi:=(\phi\times\phi)^{-1}(G')$. Relation preservation~(M3) forces $G\subseteq P_\phi$. Then for every $B'\in\mathcal A'$ (both objects being admissible),
\[
\mu^{\otimes2}\bigl((\phi^{-1}(B')\times X)\cap P_\phi\bigr)
=\mu^{\otimes2}\bigl((\phi^{-1}(B')\times X)\cap G\bigr)
=\frac{\mu'(B')}{1-\eta},
\]
and consequently the excess of the pulled--back relation carries no load:
\[
\mu^{\otimes2}\bigl((\phi^{-1}(B')\times X)\cap(P_\phi\setminus G)\bigr)=0.
\]
\end{proposition}

\begin{proof}
Both objects are admissible. By~(M4), $\mu'^{\otimes2}((B'\times X')\cap G')=\mu^{\otimes2}\bigl((\phi\times\phi)^{-1}((B'\times X')\cap G')\bigr)$; since $(\phi\times\phi)^{-1}(B'\times X')=\phi^{-1}(B')\times X$, the right side is $\mu^{\otimes2}((\phi^{-1}(B')\times X)\cap P_\phi)$. By universal rigidity (Corollary~\ref{cor:rigidity}) on $\mathcal M'$, the left side is $\mu'(B')/(1-\eta)$, giving the first equality. Applying rigidity to $\mathcal M$ at $B=\phi^{-1}(B')\in\mathcal A$ together with~(M4) gives $\mu^{\otimes2}((\phi^{-1}(B')\times X)\cap G)=\mu(\phi^{-1}(B'))/(1-\eta)=\mu'(B')/(1-\eta)$, the second. Since $G\subseteq P_\phi$ (M3), $(\phi^{-1}(B')\times X)\cap G\subseteq(\phi^{-1}(B')\times X)\cap P_\phi$ (both in $\mathcal A\otimes\mathcal A$ by~(M1)), and finite subtractivity of $\mu^{\otimes2}$ yields $\mu^{\otimes2}((\phi^{-1}(B')\times X)\cap(P_\phi\setminus G))=0$.
\end{proof}

\begin{remark}[What rigidity yields here]\label{rem:transport-rigidity}
Relation preservation~(M3) forces $G\subseteq P_\phi=(\phi\times\phi)^{-1}(G')$ automatically, so the inclusion $(\phi\times\phi)^{-1}(G')\subseteq G$ would give equality $P_\phi=G$ and is \emph{not} a genuinely weaker case. The informative content is therefore the \emph{excess} $P_\phi\setminus G$: a morphism may pull the target relation $G'$ back to a relation strictly larger than $G$, yet that excess is invisible to the load over every pulled--back test column. This is the morphism--level shadow of the null--extension philosophy (Theorem~\ref{thm:null-extension}) --- relational data beyond $G$ that the load does not see. Exact pullback $P_\phi=G$ is the special case (e.g.\ when $\mathcal M$ is the full pullback of $\mathcal M'$); in general only $G\subseteq P_\phi$ holds, with the difference load--null. No surjectivity of $\phi$ is used.
\end{remark}

\subsection{Idempotent endomorphisms}

\begin{proposition}[Retraction as split idempotent]\label{prop:retraction-idempotent}
$\Pi_R=\iota_R\circ\rho_R$, where $\rho_R:X\to R$ is $\Pi_R$ with restricted codomain and $\iota_R:R\hookrightarrow X$, with $\rho_R\circ\iota_R=\mathrm{id}_R$ (Axiom~I).
\end{proposition}

\begin{proof}
Definitional; $\rho_R\circ\iota_R=\mathrm{id}_R$ is $\Pi_R|_R=\mathrm{id}_R$.
\end{proof}

\begin{proposition}[$G$ as idempotent endorelation]\label{prop:G-idempotent}
$G\circ G=G$ (Axiom~II), so $G$ is an equivalence relation \cite{Ore} (Proposition~\ref{prop:elementary}(iii)).
\end{proposition}

\subsection{The identity--retraction subcategory and the core}

\begin{definition}[Identity--retraction subcategory]\label{def:struct-id}
$\mathbf{Struct}^{\mathrm{id}}_\eta\subseteq\mathbf{Struct}_\eta$ is the full subcategory of models with $\Pi_R=\mathrm{id}_X$ (so $R=X$, $I=\varnothing$).
\end{definition}

\begin{definition}[Identity--retraction core]\label{def:core}
For an admissible $\mathcal M$, its \emph{core} is
\[
\mathcal M|_R:=\bigl(R,\ \mathcal A\cap\mathcal P(R),\ \mu|_R,\ \mu^{\otimes2}|_{R\times R},\ R,\ \varnothing,\ \mathrm{id}_R,\ G|_R,\ \mu(R),\ \eta\bigr),
\qquad G|_R:=G\cap(R\times R).
\]
\end{definition}

\begin{remark}[Trace product algebra]\label{rem:trace-algebra}
Since $R\in\mathcal A$, write $\mathcal A_R:=\mathcal A\cap\mathcal P(R)$ for the trace algebra on $R$ (and similarly $\mathcal A'_{R'}$ for a second model). The product algebra on $\mathcal A_R$ coincides with the trace of $\mathcal A\otimes\mathcal A$ on $R\times R$:
\[
\mathcal A_R\otimes\mathcal A_R=\{\,S\cap(R\times R):S\in\mathcal A\otimes\mathcal A\,\}.
\]
Indeed a generating rectangle $B_1\times B_2$ with $B_i\in\mathcal A_R$ equals $(B_1\times B_2)\cap(R\times R)$ with $B_1\times B_2\in\mathcal A\otimes\mathcal A$, and conversely any $S\in\mathcal A\otimes\mathcal A$ is a finite Boolean combination of rectangles $B_1\times B_2$ ($B_i\in\mathcal A$), whose trace is the same combination of $(B_1\cap R)\times(B_2\cap R)$ with $B_i\cap R\in\mathcal A_R$. In particular $G|_R=G\cap(R\times R)\in\mathcal A_R\otimes\mathcal A_R$ is measurable in the core, and $\mu^{\otimes2}|_{R\times R}$ is a product charge for $\mu|_R$.
\end{remark}

The collapse theorem makes the core admissible \emph{unconditionally}, with no fiber measurability or regularity hypotheses.

\begin{theorem}[Global restriction to the core]\label{thm:null-extension}
Let $\mathcal M\in\mathbf{Struct}_\eta$. Then:
\begin{enumerate}[label=(\roman*),nosep]
\item $\mu(X\setminus R)=0$, and for every $B\in\mathcal A$,
\[
\mu^{\otimes2}\bigl((B\times X)\cap G\bigr)=\mu^{\otimes2}\bigl(((B\cap R)\times R)\cap G|_R\bigr).
\]
\item The core $\mathcal M|_R$ is an admissible structural model, i.e.\ $\mathcal M|_R\in\mathbf{Struct}^{\mathrm{id}}_\eta$.
\item The assignment $\rho:\mathbf{Struct}_\eta\to\mathbf{Struct}^{\mathrm{id}}_\eta$, $\mathcal M\mapsto\mathcal M|_R$, $\phi\mapsto\phi|_R$, is a functor, and $\rho\circ\iota=\mathrm{id}$ on $\mathbf{Struct}^{\mathrm{id}}_\eta$, where $\iota$ is the inclusion. Thus $\rho$ is a retraction of categories.
\end{enumerate}
In words: every admissible model is a null extension of its identity--retraction core; the restriction $\mathcal M\mapsto\mathcal M|_R$ discards only $\mu$--null and $\mu^{\otimes2}$--null data.
\end{theorem}

\begin{proof}
(i) Collapse (Theorem~\ref{thm:global-collapse}) gives $\mu(X\setminus R)=0$. For measurable $S\subseteq(X\setminus R)\times X$, $\mu^{\otimes2}(S)\le\mu(X\setminus R)\mu(X)=0$, and likewise for $S'\subseteq X\times(X\setminus R)$. Writing $B_R=B\cap R$ and decomposing
\[
(B\times X)\cap G=\bigl((B_R\times R)\cap G\bigr)\sqcup\bigl((B_R\times(X\setminus R))\cap G\bigr)\sqcup\bigl(((B\setminus R)\times X)\cap G\bigr),
\]
the last two terms are $\mu^{\otimes2}$--null, leaving $\mu^{\otimes2}((B\times X)\cap G)=\mu^{\otimes2}((B_R\times R)\cap G)=\mu^{\otimes2}((B_R\times R)\cap G|_R)$, since $B_R\subseteq R$.

(ii) \emph{Axiom~I} on $\mathcal M|_R$ is trivial ($\Pi_{R|_R}=\mathrm{id}_R$). \emph{Axiom~II}: $G|_R$ is reflexive ($\Delta_R\subseteq G$), symmetric, and idempotent — $G|_R\circ G|_R\subseteq(G\circ G)\cap(R\times R)=G|_R$, and reflexivity gives the reverse inclusion. \emph{Axiom~III(a)}: $\mu(R)=E_0>0$ by collapse. \emph{Axiom~III(b)}: trivial. \emph{Axiom~III(c)}: fix $B_R\in\mathcal A\cap\mathcal P(R)$. By part~(i),
\[
\mu^{\otimes2}\bigl((B_R\times R)\cap G|_R\bigr)=\mu^{\otimes2}\bigl((B_R\times X)\cap G\bigr).
\]
We apply rigidity \emph{only to $\mathcal M$}, which is admissible by hypothesis: Corollary~\ref{cor:rigidity} for $\mathcal M$ at $B=B_R$ gives $\mu^{\otimes2}((B_R\times X)\cap G)=\mu(B_R)/(1-\eta)$. (We do \emph{not} invoke rigidity for $\mathcal M|_R$, whose admissibility is what is being proved.) Combining,
\[
(1-\eta)\,\mu^{\otimes2}\bigl((B_R\times R)\cap G|_R\bigr)=\mu(B_R),
\]
which is precisely Axiom~III(c) for the identity--retraction core (where $\Pi_{R|_R}^{-1}(B_R)=B_R$). Hence $\mathcal M|_R\in\mathbf{Struct}^{\mathrm{id}}_\eta$.

(iii) For a morphism $\phi:\mathcal M\to\mathcal M'$, $\phi(R)\subseteq R'$ (Remark~\ref{rem:morphism-consequences}), so $\phi|_R:R\to R'$ is well defined. Conditions (M2)--(M3) restrict immediately. For the set--level part of (M1), $(\phi|_R)^{-1}(B')=R\cap\phi^{-1}(B')\in\mathcal A_R$ for $B'\in\mathcal A'_{R'}$. For the product--level part, given $S'\in\mathcal A'_{R'}\otimes\mathcal A'_{R'}$ choose $T'\in\mathcal A'\otimes\mathcal A'$ with $S'=T'\cap(R'\times R')$ (Remark~\ref{rem:trace-algebra}); then
\[
(\phi|_R\times\phi|_R)^{-1}(S')=(R\times R)\cap(\phi\times\phi)^{-1}(T')\in\mathcal A_R\otimes\mathcal A_R,
\]
again by Remark~\ref{rem:trace-algebra}. For (M4) we use collapse. For $B'\in\mathcal A'_{R'}$, $\phi^{-1}(B')\setminus R\subseteq X\setminus R$ is $\mu$--null by Theorem~\ref{thm:global-collapse}, so
\[
\mu_R\bigl((\phi|_R)^{-1}(B')\bigr)=\mu(R\cap\phi^{-1}(B'))=\mu(\phi^{-1}(B'))=\mu'(B')=\mu'_{R'}(B'),
\]
the third equality by collapse and the fourth by (M4) for $\phi$. The product--charge identity is analogous: for $S'\in\mathcal A'_{R'}\otimes\mathcal A'_{R'}$, the part of $(\phi\times\phi)^{-1}(T')$ outside $R\times R$ lies in $((X\setminus R)\times X)\cup(X\times(X\setminus R))$, which is $\mu^{\otimes2}$--null, so
\[
\mu^{\otimes2}_{R\times R}\bigl((\phi|_R\times\phi|_R)^{-1}(S')\bigr)=\mu^{\otimes2}\bigl((\phi\times\phi)^{-1}(T')\bigr)=\mu'^{\otimes2}(T')=\mu'^{\otimes2}_{R'\times R'}(S').
\]
The last equality uses the \emph{target} collapse as well: $T'\setminus(R'\times R')$ is contained in $((X'\setminus R')\times X')\cup(X'\times(X'\setminus R'))$, which is $\mu'^{\otimes2}$--null by Theorem~\ref{thm:global-collapse} applied to $\mathcal M'$, so $\mu'^{\otimes2}(T')=\mu'^{\otimes2}(T'\cap(R'\times R'))=\mu'^{\otimes2}_{R'\times R'}(S')$. Composition and identity are preserved, so $\rho$ is a functor. For $\mathcal M\in\mathbf{Struct}^{\mathrm{id}}_\eta$, $R=X$ and $\mathcal M|_R=\mathcal M$, so $\rho\circ\iota=\mathrm{id}$.
\end{proof}

\begin{corollary}[Closed form of the load on any model]\label{cor:load-closed-form}
For every admissible $\mathcal M$ and every $B\in\mathcal A$, $\mu^{\otimes2}((B\times X)\cap G)=\mu(B)/(1-\eta)$. If moreover the core $\mathcal M|_R$ satisfies (K--fin) or (K--ctbl), then, writing $B_R:=B\cap R$,
\[
\mu^{\otimes2}((B\times X)\cap G)=\sum_{k\in K}\mu(B_R\cap C_k)\,\mu(C_k),\qquad \mu(C_k)\in\{0,(1-\eta)^{-1}\}.
\]
\end{corollary}

\begin{proof}
The first identity is Corollary~\ref{cor:rigidity}. The block form follows from Theorem~\ref{thm:null-extension}(i) and Theorem~\ref{thm:pi-id-classification} applied to $\mathcal M|_R$.
\end{proof}

\begin{remark}[Place of the factorization]\label{rem:place-factorization}
Theorem~\ref{thm:null-extension} replaces the fiber--measurability quotient construction of the earlier formulation. Because collapse holds for every admissible model, no regularity hypothesis (R--fin)/(R--ctbl) on fibers is needed to define the core or to make $\rho$ a functor; such hypotheses survive only where the \emph{block formula} (an internal property of the core) is invoked, i.e.\ in Theorem~\ref{thm:pi-id-classification}.
\end{remark}

\subsection{The fiberwise statement and why it is subsumed}\label{sec:fiber-free}

The reduction of Theorem~\ref{thm:null-extension} rests on the single global identity $\mu(X\setminus R)=0$, obtained from the $B=R$ instance of invariance. It is instructive to compare this with the \emph{fiberwise} route, which proves the same conclusion locally and then sums it under cardinality and additivity hypotheses. We record the fiberwise statement, show it is the $r$--local shadow of collapse, and make precise the hypotheses that collapse renders unnecessary.

\begin{definition}[Fiber partition]\label{def:fibers}
For $r\in R$, the \emph{fiber} of $\Pi_R$ over $r$ is $F_r:=\Pi_R^{-1}(\{r\})$. Under Axiom~I, $\Pi_R|_R=\mathrm{id}_R$ gives $r\in F_r$, and $X=\bigsqcup_{r\in R}F_r$.
\end{definition}

\begin{lemma}[Fiberwise annihilation]\label{lem:fiberwise}
Assume Axiom~I and Axiom~III(b), and suppose $\{r\}\in\mathcal A$ and $F_r\in\mathcal A$ for a given $r\in R$. Then
\[
\mu(F_r\setminus\{r\})=0.
\]
\end{lemma}

\begin{proof}
Since $r\in F_r$, the union $F_r=\{r\}\sqcup(F_r\setminus\{r\})$ is disjoint and lies in $\mathcal A$. Invariance at $B=\{r\}$ gives $\mu(F_r)=\mu(\Pi_R^{-1}(\{r\}))=\mu(\{r\})$, so finite additivity yields $\mu(\{r\})+\mu(F_r\setminus\{r\})=\mu(\{r\})$, hence $\mu(F_r\setminus\{r\})=0$.
\end{proof}

\begin{proposition}[Fiber--free reduction: collapse subsumes the fiberwise route]\label{prop:fiber-free}
Lemma~\ref{lem:fiberwise} is the $r$--local shadow of global collapse. A hypothesis--based derivation of $\mu(X\setminus R)=0$ would proceed by
\[
X\setminus R=\bigsqcup_{r\in R}(F_r\setminus\{r\})
\]
and summing Lemma~\ref{lem:fiberwise} over $r$, which requires:
\begin{itemize}[nosep]
\item fiber measurability $\{r\},F_r\in\mathcal A$ for \emph{every} $r\in R$; and
\item a summability hypothesis, either \emph{(R--fin)} $R$ finite, or \emph{(R--ctbl)} $R$ countable with $\mu$ $\sigma$--additive,
\end{itemize}
since finite additivity alone does not control an uncountable disjoint union of null sets. The collapse theorem (Theorem~\ref{thm:global-collapse}) obtains the global conclusion $\mu(X\setminus R)=0$ with \emph{none} of these hypotheses: it uses only $R\in\mathcal A$ (so $X\setminus R\in\mathcal A$) and the single instance $B=R$ of invariance. Consequently the core construction (Definition~\ref{def:core}) and the functor $\rho$ (Theorem~\ref{thm:null-extension}) are defined on all of $\mathbf{Struct}_\eta$, with no fiber subcategory.
\end{proposition}

\begin{proof}
The displayed decomposition is the fiber partition (Definition~\ref{def:fibers}) minus the representatives. Under (R--fin) the union is finite and finite additivity applies; under (R--ctbl) it is countable and $\sigma$--additivity applies; either way $\mu(X\setminus R)=\sum_{r}\mu(F_r\setminus\{r\})=0$ by Lemma~\ref{lem:fiberwise}. Without such a hypothesis the sum is uncontrolled, as a finitely additive measure may assign positive mass to an uncountable disjoint union of null sets. That collapse avoids the issue entirely is the content of Theorem~\ref{thm:global-collapse}, whose proof takes $B=R$ directly: $\mu(X)=\mu(\Pi_R^{-1}(R))=\mu(R)$ and $\mu(X)=\mu(R)+\mu(X\setminus R)$ force $\mu(X\setminus R)=0$.
\end{proof}

\begin{remark}[What the comparison buys]\label{rem:fiber-comparison}
Proposition~\ref{prop:fiber-free} isolates the precise mechanism of the simplification: invariance is a constraint on $\mu\circ\Pi_R^{-1}$, and its strongest single instance is $B=R$, which already sees \emph{all} of $X$ at once (because $\Pi_R^{-1}(R)=X$). The fiberwise instances $B=\{r\}$ see only one fiber each, so recombining them costs a summability hypothesis. The lesson is that the global instance is not merely more efficient but strictly more powerful in the finitely additive setting: it reaches conclusions (uncountable $R$, no $\sigma$--additivity) that the fiberwise route cannot.
\end{remark}


\section{Examples}\label{sec:examples}

\begin{example}[Representative core with a null extension]\label{ex:repr-null}
Let $X$ be finite, partitioned into nonempty blocks $(C_j)_{j\in J}$. Choose $r_j\in C_j$, set $R=\{r_j:j\in J\}$, $I=X\setminus R$, $\mathcal A=\mathcal P(X)$, $\Pi_R(x)=r_j$ for $x\in C_j$, and $G=\bigcup_j(C_j\times C_j)$. Choose $w_j\in\{0,1\}$ with $\sum_j w_j>0$ and $\mu(B)=\sum_{j:r_j\in B}w_j$, $E_0=\mu(X)$, $\eta=0$. By the verification in Theorem~\ref{thm:finite-model}'s style this is admissible; consistently with collapse, $\mu(I)=0$, and the core $\mathcal M|_R$ is the identity--retraction model on the representatives. The blocks $C_j$ form a null extension of the core: they carry relational structure through $G$ but no $\mu$--mass.
\end{example}

\begin{example}[Quotient weights and forced representative support]\label{ex:quotient-weights}
This example is not an external application but a \emph{diagnostic} use of the axioms: it shows how admissibility decides which weightings of a quotient structure are allowed. Let $Q=\{1,\dots,n\}$ be a finite set of types. For each $q\in Q$ let
\[
C_q=\{r_q,\,h_{q,1},\dots,h_{q,m_q}\}
\]
be finitely many descriptions of the same type, with distinguished representative $r_q$. Set
\[
X=\bigsqcup_{q\in Q}C_q,\quad R=\{r_q:q\in Q\},\quad I=X\setminus R,\quad \mathcal A=\mathcal P(X),
\]
$\Pi_R(x)=r_q$ for $x\in C_q$, and $G=\bigsqcup_{q\in Q}(C_q\times C_q)$, so $G$ identifies descriptions of the same type. The $C_q$ are exactly the $G$--equivalence classes, all measurable.

\emph{Step 1 (support, from invariance alone).} Suppose one wishes to weight all descriptions by a finitely additive $\mu$ subject to projection--measure invariance $\mu(\Pi_R^{-1}(B))=\mu(B)$ for $B\subseteq R$ (subclause~III(b)). Already its instance $B=R$ gives $\mu(X)=\mu(\Pi_R^{-1}(R))=\mu(R)$, hence by finite additivity (cf.\ Theorem~\ref{thm:global-collapse})
\[
\mu(X\setminus R)=0,\qquad\text{equivalently}\qquad \mu(C_q\setminus\{r_q\})=0\quad(q\in Q).
\]
The redundant descriptions $h_{q,i}$ may persist as elements of $X$ and as members of $G$, but they cannot carry positive $\mu$--mass: the quotient is measured entirely on the representative sector $R$. This step needs only III(b), not the full axiom system.

\emph{Step 2 (quantization, $\eta\neq0$).} If these data are completed to an admissible structural model with parameter $\eta\in(0,1)$, then the already--imposed invariance~III(b) and the coupling law~III(c), together with Axiom~II, give the full rigidity package. In particular the load is no longer free,
\[
\mu^{\otimes2}((B\times X)\cap G)=\frac{\mu(B)}{1-\eta}\qquad(B\in\mathcal A),
\]
and, each $C_q$ being a measurable $G$--equivalence class, Theorem~\ref{thm:component-quantization} forces
\[
\mu(C_q)\in\Bigl\{0,\tfrac{1}{1-\eta}\Bigr\}.
\]
Since $\mu(C_q\setminus\{r_q\})=0$ by Step 1, this is the same as $\mu(\{r_q\})\in\{0,(1-\eta)^{-1}\}$: a type is either $\mu$--invisible or carries exactly the quantum $(1-\eta)^{-1}$, concentrated on its representative. One cannot, for instance, build an admissible model in which the representative of one type weighs twice that of another.

\emph{Step 3 (a counting constraint).} Let $N$ be the number of positive--mass types. The positive--mass classes are finitely many, measurable, and their union is $\mu$--conull (the rest being $\mu$--null by Steps 1--2), so Corollary~\ref{cor:integrality} applies and yields
\[
N=(1-\eta)\,E_0\in\mathbb{Z}_{>0}.
\]
Thus the axioms act as an admissibility test for quotient weights: redundant descriptions cannot be counted, observable types cannot be weighted arbitrarily, and the global scale $E_0$ and rate $\eta$ are tied to one another through the integer $N$. A would--be model with, say, $E_0=3$ and $\eta=\tfrac12$ is inadmissible in this block form, since $(1-\eta)E_0=\tfrac32\notin\mathbb{Z}$.
\end{example}


\section{Conclusion}\label{sec:conclusion}

We have studied a minimal ZFC--internal axiom system for admissible structural models, pre-structural data $(X,\mathcal A,\mu,\mu^{\otimes2},R,I,\Pi_R,G,E_0,\eta)$ satisfying Axioms~I,~II,~III.

The central phenomenon is not that the coupling law admits a fixed--point representation, but that admissibility forces a global collapse:
\[
\mu(X\setminus R)=0\qquad\text{(Theorem~\ref{thm:global-collapse})}.
\]
Every admissible model is therefore, modulo null sets, an identity--retraction model (Theorem~\ref{thm:null-extension}). The coupling law then rigidly determines the two--point load,
\[
\mu^{\otimes2}\bigl((B\times X)\cap G\bigr)=\frac{\mu(B)}{1-\eta}\qquad\text{(Corollary~\ref{cor:rigidity})},
\]
the endpoint $\eta=1$ is automatically excluded with feasibility bound $E_0\ge(1-\eta)^{-1}$ (Propositions~\ref{prop:eta-lt-one},~\ref{prop:feasibility-bound}), and every measurable $G$--component has mass $0$ or $(1-\eta)^{-1}$ (Theorem~\ref{thm:component-quantization}).

The independence results (Theorem~\ref{thm:independence-main}, Corollary~\ref{cor:minimality}) show that this collapse is not the product of redundant axioms: the three axioms, and the three subclauses of Axiom~III, are mutually independent, yet together they force the collapse and the surrounding rigidity. The fixed--point reformulation (Theorem~\ref{thm:coupling-fixed-point}) survives as a corollary, valid for arbitrary bounded charges and recovering the closed form as the unique fixed point of a contraction. Under $\sigma$--additive probability normalization the global identity forces $\eta=0$ (Section~\ref{sec:normalization}), but we have shown this obstruction is finitely additive in nature, with $\sigma$--additivity entering only through product--measure uniqueness and countable block formulas.

\paragraph{Further directions.}
Natural extensions include: replacing $\mathcal A$ by a $\sigma$--algebra under hypotheses giving a canonical product charge; enriching $G$ to a weighted kernel $K:X\times X\to[0,\infty)$ and comparing with graph limits \cite{Lovasz}; an adjoint characterization of the core retraction $\rho:\mathbf{Struct}_\eta\to\mathbf{Struct}^{\mathrm{id}}_\eta$ \cite{MacLane,Riehl}; operator--theoretic variants replacing $\mu$ by a state or trace; and large--mass scaling limits. For the last, note that the collapse theorem keeps every \emph{individual} admissible model finite--mass, so the relevant object is a sequence $\mathcal M_n$ of admissible models with finite total masses $E_{0,n}=\mu_n(X_n)\to\infty$; the limiting behavior of $\eta_n$ and of the rescaled load is then governed by local instances of the coupling law rather than by any single infinite--mass model.


\end{document}